\documentclass[11pt]{amsart}
\usepackage[margin=1in]{geometry}
\usepackage[utf8]{inputenc}
\usepackage{amsthm, amssymb}
\usepackage[colorlinks=true, pdfstartview=FitV, linkcolor=blue, citecolor=blue, urlcolor=blue]{hyperref}
\usepackage[colorinlistoftodos]{todonotes}
\usepackage{algorithm}
\usepackage{algpseudocode}
\usepackage{pgf}
\usepackage{tikz}
\usepackage{tikz-cd}
\usetikzlibrary{positioning,shapes,shadows,arrows}
\usepackage{bbm,bm}
\usepackage{enumitem}
\usepackage{mathabx}
\usepackage{comment}
\usepackage{ytableau}
\usepackage{thmtools}
\usepackage{thm-restate}

\theoremstyle{definition}
\newtheorem{prop}{Proposition}[section]
\newtheorem{ourthm}[prop]{Theorem}
\newtheorem{lem}[prop]{Lemma}
\newtheorem{lemma}[prop]{Lemma}
\newtheorem{cor}[prop]{Corollary}

\theoremstyle{remark}
\newtheorem{exa}[prop]{Example}
\newtheorem{rem}[prop]{Remark}
\newtheorem{defn}[prop]{Definition}
\newtheorem{prob}[prop]{Problem}

\newcommand{\NN}{\mathbb{N}}

\newcommand{\fG}{\mathfrak{G}}
\newcommand{\fGb}{\fG}
\newcommand{\fL}{\mathfrak{L}}
\newcommand{\fLb}{\fL}
\newcommand{\topLas}{\widehat{\fL}}
\newcommand{\topGro}{\widehat{\fG}}
\newcommand{\fS}{\mathfrak{S}}

\newcommand{\KKD}{\mathsf{KKD}}

\newcommand{\X}{\mathsf{X}}

\newcommand{\rev}{\mathsf{rev}}

\newcommand{\supp}{\mathsf{supp}}

\newcommand{\wt}{\mathsf{wt}}
\newcommand{\Z}{\mathbb{Z}}
\newcommand{\Q}{\mathbb{Q}}

\newcommand{\ex}{\mathsf{ex}}

\newcommand{\dark}{\mathsf{dark}}
\newcommand{\snow}{\mathsf{snow}}
\newcommand{\rajcode}{\mathsf{rajcode}}
\newcommand{\raj}{\mathsf{raj}}
\newcommand{\Inv}{\mathsf{Inv}}

\newcommand{\snowflake}{\Asterisk}

\newcommand{\invcode}{\mathsf{invcode}}
\newcommand{\inv}{\mathsf{inv}}
\newcommand{\Rook}{\mathsf{Rook}}
\newcommand{\Stair}{\mathsf{Stair}}
\newcommand{\GR}{\mathsf{GR}}
\newcommand{\NW}{\mathsf{NW}}

\newcommand{\Rowone}[1]
{{\mathsf{{Row}}}_1(P(#1))}
\newcommand{\cb}[1]{{\color{blue}{\bm{$#1$}}}}

\newcommand{\Hilb}{\mathrm{Hilb}}

\definecolor{darkblue}{rgb}{0.0,0,0.7} 
\definecolor{darkred}{rgb}{0.7,0,0} 
\definecolor{darkgreen}{rgb}{0, .6, 0} 
\newcommand{\definition}[1]{{\color{darkred}\emph{#1}}} 

\author[J.~Pan]{Jianping Pan}
\address[J. Pan]{Department of Mathematics, NC State University, Raleigh, NC 95616-8633, U.S.A.}
\email{jpan9@ncsu.edu}

\author[T.~Yu]{Tianyi Yu}
\address[T. Yu]{Department of Mathematics, UC San Diego, La Jolla, CA 92093, U.S.A.}
\email{tiy059@ucsd.edu}

\title{Top-degree components of Grothendieck and Lascoux polynomials}

\keywords{Grothendieck polynomials, Lascoux polynomials, Hilbert series, Castelnuovo–Mumford polynomials}
\subjclass[2020]{Primary 05E05}

\begin{document}

\maketitle
\begin{abstract}
The Castelnuovo–Mumford polynomial $\widehat{\mathfrak{G}}_w$ with $w \in S_n$
is the highest homogeneous component of the Grothendieck polynomial $\mathfrak{G}_w$.
Pechenik, Speyer and Weigandt define a statistic $\rajcode(\cdot)$ 
on $S_n$ that gives the leading monomial of $\widehat{\mathfrak{G}}_w$.
We introduce a statistic $\rajcode(\cdot)$ on any diagram $D$
through a combinatorial construction ``snow diagram''
that augments and decorates $D$.
When $D$ is the Rothe diagram of a permutation $w$, $\rajcode(D)$ 
agrees with the aforementioned $\rajcode(w)$.
When $D$ is the key diagram of a weak composition $\alpha$, $\rajcode(D)$ yields 
the leading monomial of $\widehat{\mathfrak{L}}_\alpha$, 
the highest homogeneous component of the Lascoux polynomials $\mathfrak{L}_\alpha$.
We use $\widehat{\mathfrak{L}}_\alpha$ to construct a basis of $\widehat{V}_n$, the span of 
$\widehat{\mathfrak{G}}_w$ with $w \in S_n$.
Then we show $\widehat{V}_n$ gives a natural algebraic 
interpretation of a classical $q$-analogue of Bell numbers.
\end{abstract}

\section{Introduction}
\label{S: Intro}

Introduced by Lascoux and Sch\"utzenberger \cite{LS:Groth},
the Grothendieck polynomial $\fGb_w$
is a polynomial representative of the 
$K$-class of structure sheaves of 
Schubert varieties of flag varieties.
It is the inhomogeneous analogue
of the Schubert polynomial $\fS_w$: 
The lowest-degree component
of $\fG_w$ forms $\fS_w$.
Pechenik, Speyer and Weigandt~\cite{PSW} introduce the
Castelnuovo–Mumford polynomial $\topGro_w$\footnote{Pechenik, Speyer and Weigandt~\cite{PSW} denote it as $\mathfrak{CM}_w$.}, 
the top-degree component of $\fG_w$.
They describe the leading monomial of $\topGro_w$
with respect to the tail lexicographic order
by defining a new statistic 
$\rajcode(\cdot)$ on $S_n$.
We summarize some of their results on $\topGro_w$.

\begin{ourthm}[~\cite{PSW}]
\label{thm:psw}
Let $w,u$ be permutations in $S_n$.
\begin{enumerate}
\item[(A)]
The polynomial $\topGro_w$ has leading monomial $x^{\rajcode(w)}$.
\item[(B)] 
We have
$\topGro_w$ is a scalar multiple of $\topGro_{u}$
if and only if $\rajcode(w) = \rajcode(u)$.
\item[(C)] 
If $w$ is inverse fireworks 
(see \S\ref{S: permutations}), 
then $x^{\rajcode(w)}$ has 
coefficient 1 in $\topGro_w$.
Moreover, there exists
exactly one $u'\in S_n$ that is inverse fireworks
such that $\rajcode(u) = \rajcode(u')$.
\end{enumerate}
\end{ourthm}

Dreyer, M\'esz\'aros and St. Dizier~\cite{dms22}
provide an alternative proof of (A) via the 
climbing chain model for Grothendieck polynomials
introduced by Lenart, Robinson, and Sottile~\cite{LRS}.
Hafner~\cite{Haf} provides an alternative proof of (A)
for vexillary permutations via bumpless \allowbreak pipedreams.

Schubert polynomials are related
to key polynomials $\kappa_\alpha$
which are indexed by weak compositions.
The key polynomials are the characters of 
Demazure modules~\cite{De}.
Both Schubert and key polynomials 
can be defined recursively 
via the divided difference operators 
(see \S\ref{S: Background}).
In addition, Schubert polynomials
expand positively into key polynomials~\cite{RS}.
The key polynomials also have
inhomogeneous analogues called Lascoux polynomials $\fLb_\alpha$~\cite{Las}.
Grothendieck polynomials and Lascoux polynomials are related: 
An expansion of Grothendieck polynomials into Lascoux polynomials 
was conjectured by Reiner and Yong~\cite{RY} and 
proven by Shimozono and Yu~\cite{SY}.

Due to the connection between $\fGb_w$ and $\fLb_\alpha$, one would expect the top Lascoux polynomial $\topLas_\alpha$, the top-degree component of $\fLb_\alpha$, to parallel $\topGro_w$. 
We define a statistic $\rajcode(\cdot)$ on weak compositions and show in \S\ref{S: weak composition} that $\topLas_\alpha$ enjoy properties analogous to the properties of $\topGro_w$ listed in Theorem~\ref{thm:psw}:

\begin{restatable}{ourthm}{topLasThm}
\label{T: Top Las}
Let $\alpha$ and $\gamma$ be two weak compositions. 
\begin{enumerate}
\item[(a)]
The polynomial $\topLas_\alpha$ has leading monomial $x^{\rajcode(\alpha)}$.
\item[(b)] We have
$\topLas_\alpha$ is a scalar multiple of $\topLas_{\gamma}$
if and only if $\rajcode(\alpha) = \rajcode(\gamma)$.
\item[(c)] 
We say $\alpha$ is snowy if its
positive entries are distinct.
If $\alpha$ is snowy, then $x^{\rajcode(\alpha)}$ has coefficient 1 in $\topLas_\alpha$.
Moreover, 
there exists exactly one snowy weak composition $\gamma'$ 
such that $\rajcode(\gamma) = \rajcode(\gamma')$.
\end{enumerate}  
\end{restatable}

Our definition of $\rajcode(\cdot)$
on weak compositions is diagrammatic. 
Given a diagram $D$, 
we define a combinatorial
construction called the snow diagram 
that augments and decorates $D$.
Let $\rajcode(D)$
be the weight of the snow diagram.
Every weak composition $\alpha$ is naturally
associated with a diagram called
the key diagram $D(\alpha)$
(see Subsection~\ref{subsec.diagrams}).
Then we define $\rajcode(\alpha):=\rajcode(D(\alpha))$.

Snow diagrams unify the computation of leading monomials in $\topGro_w$ and $\topLas_\alpha$.
Each permutation $w$ is also associated
with a diagram called the Rothe diagram
$RD(w)$.
In \S\ref{S: permutations},
we show $\rajcode(w) = \rajcode(RD(w))$.
In other words, we give a diagrammatic way to compute $\rajcode(w)$.

Finally, let $\widehat{V}_n :=  \Q\textrm{-span}\{\topGro_w: w \in S_n\}$ and $\widehat{V} :=  
\bigcup_{n \geq 1} \widehat{V}_n$.
In Proposition~\ref{P: Filtered Algebra},
we show $\widehat{V}$
is a filtered algebra. 
Theorem~\ref{thm:psw} can be used to construct a basis of $\widehat{V}_n$
and $\widehat{V}$ 
consisting of $\topGro_w$.
In particular, 
the dimension of $\widehat{V}_n$
is $B_n$, the $n\textsuperscript{th}$ Bell number.
In \S\ref{S: Space},
we use Theorem~\ref{T: Top Las} to construct another basis consisting of $\topLas_\alpha$.
This basis allows us to compute the Hilbert series of $\widehat{V}_n$ 
and $\widehat{V}$
involving a $q$-analogue of $B_n$.

The rest of the paper is organized as follows. 
In \S\ref{S: Background}, we provide necessary background information and notation.
In \S\ref{S: Cloud and snow}, we construct a snow diagram from any diagram 
and define statistics $\rajcode(\cdot)$ and $ \raj(\cdot)$ on all diagrams.
In \S\ref{S: weak composition}, we  prove Theorem~\ref{T: Top Las}.
In \S\ref{S: permutations}, we show the statistics $\rajcode(\cdot)$ and $\raj(\cdot)$ 
on a Rothe diagram are equivalent to that defined in~\cite{PSW}.
We also relate the snow diagram to two classical constructions: 
Schensted insertion and the shadow diagram.
In \S\ref{S: Space}, we derive the Hilbert series 
of $\widehat{V}_n$ and $\widehat{V}$.
In \S\ref{sec.open}, we present several open problems and future directions.

\section{Background}
\label{S: Background}
\subsection{Polynomials}
We provide necessary background
for Grothendieck polynomials and Lascoux polynomials.
Then we introduce
$\topGro_w$ and $\topLas_\alpha$ 
which span the spaces $\widehat{V}_n$ and $\widehat{V}$. 

The \definition{Grothendieck polynomials} $\fGb_w \in 
\mathbb{Z}_{\geqslant 0}[x_1, x_2, \dots][\beta]$
were recursively defined by Lascoux and Sch{\"u}tzenberger~\cite{LS:Groth}. 
Let $\partial_i(\cdot)$ be the divided difference operators acting 
on the polynomial ring.
For each $i$,
define $\partial_i(f) := \dfrac{f - s_i f}{x_i-x_{i+1}}$,
where $s_i$ is the operator that swaps $x_i$ and $x_{i+1}$.
Then for $w\in S_n$, 
\begin{align*}
\fGb_w &:=
\begin{cases}
x_1^{n-1}x_2^{n-2}\cdots x_{n-1} &\text{if $w$ is $[n, n-1, \dots, 1]$ in one-line notation,} \\
\partial_i( (1 + \beta x_{i+1}) \fGb_{ws_i} ) &\text{if $w(i)<w(i+1)$.}
\end{cases}
\end{align*}

Let $S_+$ be the set of permutations
of $\{1, 2, \dots\}$
such that only finitely many numbers are permuted.
Take $w \in S_+$ and assume $w$ only permutes numbers in $[n]$.
Let $w' \in S_n$ be the restriction of $w$ to $[n]$
and define $\fGb_w$ as $\fGb_{w'}$.
It is shown in~\cite{LS:Groth} that $\fGb_w$ is well-defined.

A \definition{weak composition}
is an infinite sequence of non-negative integers with finitely many positive entries. 
Let $C_+$ be the set of weak compositions. 
For $\alpha \in C_+$, 
we use $\alpha_i$ to denote its $i$\textsuperscript{th} entry,
and write $\alpha = (\alpha_1, \alpha_2, \dots, \alpha_n)$ 
where $\alpha_n$ is the last positive entry. 
We use $x^\alpha$ to denote the monomial
$x_1^{\alpha_1}x_2^{\alpha_2}\cdots x_n^{\alpha_n}$ and
$|\alpha| = \sum_{i\geqslant 1}^n\alpha_i$.
The \definition{Lascoux polynomials} 
$\fLb_\alpha$, indexed by weak compositions,
are in $\mathbb{Z}_{\geqslant 0}[x_1, x_2, \dots][\beta]$. 
By~\cite{Las}, they are defined recursively
\begin{align*}
\fLb_\alpha = \begin{cases}
x^\alpha & \text{if $\alpha$ is weakly decreasing,} \\
\pi_i ((1 + \beta x_{i+1}) \fLb_{s_i\alpha} ) &\text{if $\alpha_i<\alpha_{i+1}$,}
\end{cases}
\end{align*}
where $\pi_i$ is the operator $\pi_i(f) := \partial_i(x_i f)$.

We say a pair $(i,j)$ is an \definition{inversion}
of $w \in S_n$ if $i < j$ and $w(i) > w(j)$. 
Let $\Inv(w)$ be the set of all inversions in $w$
and let $\inv(w) = |\Inv(w)|$.
Then we may view $\fGb_w$ as a polynomial in $\beta$, where
\[
[\beta^d] \fGb_w := \textrm{coefficient of $\beta^d$ in }\fGb_w
\]
is a homogeneous polynomial in the $x$-variables with degree 
$\inv(w) + d$ in $\Z_{\geq 0}[x_1,x_2,\dots]$.
The \definition{Schubert polynomial} $\fS_w := [\beta^0]\fGb_w$.
Similarly, viewing $\fLb_\alpha$ as a polynomial of $\beta$,
$[\beta^d] \fLb_\alpha$ is a homogeneous polynomial with degree
$|\alpha| + d$ in $\Z_{\geq 0}[x_1,x_2,\dots]$.
The \definition{key polynomial} $\kappa_\alpha := [\beta^0]\fLb_\alpha$.
The representation theoretic, geometric and combinatorial perspectives of Schubert polynomials and key polynomials
are well-studied~\cite{De,LS1,BJS}.

Define $V_n :=  \Q\textrm{-span}\{\fS_w: w \in S_n\}$ 
and $V :=  \Q\textrm{-span}\{\fS_w: w \in S_+\} = \bigcup_{n \geq 1} V_n$.
In fact, $V = \mathbb{Q}[x_1, x_2, \dots]$. 
By the increasing sequence 
$V_1 \subset V_2 \subset \cdots \subset V$,
$V$ has the structure of a filtered algebra.

In this paper, we are interested in the top-degree 
components of $\fGb_w$ and $\fLb_\alpha$. 
For a polynomial $f \in \mathbb{Q}[x_1, x_2, \dots][\beta]$,
let $\widehat{f} = [\beta^d](f)$ where $d$ is the largest such that 
$[\beta^d](f) \neq 0$.
The \definition{Castelnuovo–Mumford polynomial} of $w \in S_+$
is defined as $\topGro_w$.
The \definition{top Lascoux polynomial} of $\alpha \in C_+$
is defined as $\topLas_\alpha$. In appendix \S\ref{appendix}, we list some Grothendieck polynomials and Lascoux polynomials.
Pechenik, Speyer and Weigandt~\cite{PSW} first study $\topGro_w$.
To the best of the authors knowledge, 
$\topLas_\alpha$ has not been studied previously. 

Now consider the \definition{tail lexicographic order} 
on monomials in the $x$-variables.
We say a monomial $x^\alpha$ is larger than $x^\gamma$
if there exists $k$ such that $\alpha_k > \gamma_k$ and 
$\alpha_j = \gamma_j$ for all $j > k$.
The \definition{leading monomial} of $f \in \Q[x_1, x_2, \dots]$
is the largest monomial in $f$. 
Among the four homogeneous
polynomials above,
three of them have combinatorial rules for their leading terms:
\begin{enumerate}
\item~\cite{BJS} The leading monomial of $\fS_w$ with $w \in S_n$
is $x^{\invcode(w)}$,
where $$\invcode(w)_i = |\{j: (i,j) \in \Inv(w)\}|.$$
\item~\cite{LS2} The leading monomial of $\kappa_\alpha$ is $x^\alpha$.
\item~\cite{PSW} The leading monomial of $\topGro_w$
is $x^{\rajcode(w)}$ defined as follows.
\end{enumerate}

\begin{defn}\cite{PSW}
\label{D: raj on permutations}
Let $\textup{LIS}^w(q)$ be the 
length of the longest increasing subsequence 
of $w\in S_n$ that starts with $q$.
The $\rajcode(w)$ for $w\in S_n$ is a weak composition
where 
$\rajcode(w)_r := n + 1 - r - \textup{LIS}^w(w(r))$ for $r \in [n]$
and $0$ if $r > n$.
Then $\raj(w) := |\rajcode(w)|$.
\end{defn}

\begin{exa}
\label{E: rajcode permutation}
Consider $w = 3721564 \in S_7$. We have $\textup{LIS}^w(2) = 3$, 
so $\rajcode(w)_3 = 7 + 1 - 3 - 3 = 2$.
All together,
we get $\rajcode(w) = (4,5,2,1,1,1)$ and $\raj(w) = 14$.
\end{exa}

We will define $\rajcode(\cdot)$ on $C_+$
and show the leading monomial of $\fLb_\alpha$ is $x^{\rajcode(\alpha)}$
in \S\ref{S: weak composition}.

A connection between $\fGb_w$ and $\fLb_\alpha$
is established by Shimozono and Yu~\cite{SY}.
To describe this connection, 
we need the following notion.
\begin{defn}
Let $f, f_1, f_2, \dots$ be polynomials in 
$\mathbb{Z}_{\geq 0}[x_1, x_2,\dots]$.
We say $f$ \definition{expands positively} 
into $\{f_1,f_2, \dots\}$ if there exist 
$c_1,c_2, \dots \in \mathbb{Z}_{\geq 0}$
such that $f = \sum_{i} c_if_i$.

Now assume $f, f_1, f_2, \dots$ are polynomials 
in $\mathbb{Z}_{ \geq 0}[\beta][x_1, x_2,\dots]$.
We say $f$ \definition{expands positively} 
into $\{f_1,f_2, \dots\}$ if there exist 
$g_1,g_2, \dots \in \mathbb{Z}_{\geq 0}[\beta]$ 
such that $f = \sum_{i} g_if_i$.
\end{defn}

\begin{ourthm}[\cite{SY}]
\label{T: Gro to Las}
For $w \in S_+$, $\fGb_w$ expands positively
into $\{\fLb_\alpha: \alpha \in C_+\}$. 
\end{ourthm}
 
This result implies
$\topGro_w$ also expands
positively into
$\topLas_\alpha$ by the following lemma whose proof
is sufficiently elementary. 

\begin{lemma}
\label{L: Expands positively}
Let $f, f_1, f_2, \dots$ 
in $\mathbb{Z}_{\geq 0}[\beta][x_1, x_2,\dots]$.
If $f$ expands positively 
into $\{f_1, f_2, \dots\}$,
then $\widehat{f}$ expands positively into $\widehat{f_1}, \widehat{f_2}, \dots $.
\end{lemma}

\begin{cor}
\label{C: top Gro to top Las}
For $w \in S_+$,
$\topGro_w$ expands positively into $\{\topLas_\alpha: \alpha \in C_+\}$.
\end{cor}

Define $\widehat{V}_n :=  \Q\textrm{-span}\{\topGro_w: w \in S_n\}$ and $\widehat{V} :=  \Q\textrm{-span}\{\topGro_w: w \in S_+\} = \bigcup_{n \geq 1} \widehat{V}_n$.
By work of Lascoux, Sch{\"u}tzenberger~\cite{Las2} and Brion~\cite{Bri},
the product $\fGb_u \fGb_v$ with $u \in S_m$ and $v \in S_n$
expands positively into $\fGb_w$ with $w \in S_{m + n}$.
By Lemma~\ref{L: Expands positively},
$\topGro_u \topGro_v$ with $u \in S_m$ and $v \in S_n$ expands positively into $\topGro_w$ 
with $w \in S_{m + n}$.
Finally, we conclude the following.
\begin{prop}
\label{P: Filtered Algebra}
The space $\widehat{V}$ is a filtered 
algebra with respect to the filtration
$\widehat{V}_1 \subset \widehat{V}_2 \subset \cdots \subset \widehat{V}$.
\end{prop}

\subsection{Diagrams}
\label{subsec.diagrams}
A \definition{diagram} is a finite subset 
of $\mathbb{Z}_{>0} \times \mathbb{Z}_{>0}$.
We represent a diagram by putting a cell at row $r$
and column $c$ for each $(r,c)$ in the diagram.
The leftmost column (resp. topmost row) is called column $1$ (resp. row $1$).
The \definition{weight} of a diagram $D$, 
denoted as $\wt(D)$, 
is a weak composition 
whose $i$\textsuperscript{th}  entry is the number of boxes 
in its row $i$.
We recall two classical families of diagrams. 

Each weak composition $\alpha$
is associated with a diagram called 
the \definition{key diagram}, 
denoted as $D(\alpha)$.
It is the unique left-justified diagram 
with weight $\alpha$.
One important key diagram we will use later is 
$\Stair_n := D((n-1, n-2, \cdots, 1))$.

\begin{exa}
The following are two examples of key diagrams.
For clarity, we put an ``$i$'' 
on the left of row $i$
and put a small dot in each cell.
\begin{align*}
D(0,2,1) = 
\raisebox{0.45cm}{
\begin{ytableau}
\none[1]\cr
\none[2] & \cdot & \cdot \cr
\none[3] & \cdot
\end{ytableau}}
\quad , \quad \quad \quad
\Stair_4 = 
\raisebox{0.45cm}{
\begin{ytableau}
\none[1] & \cdot & \cdot & \cdot \cr
\none[2] & \cdot & \cdot \cr
\none[3] & \cdot
\end{ytableau}} 
\end{align*}
\end{exa}

Each permutation $w$ 
is associated with
the \definition{Rothe diagram} 
$RD(w):=\{ (r, w(r'): (r,r') \in \Inv(w)\}$.

\begin{exa}
Let $w = 41532 \in S_5$.
Then $\Inv(w) = \{ (1,2), (1,4), (1,5), 
(3,4), (3,5), (4,5)\}$.
The Rothe diagram is depicted as follows.
$$
RD(w) = 
\raisebox{1cm}{
\begin{ytableau}
\none[1] & \cdot & \cdot & \cdot \cr
\none[2] & \none & \none & \none \cr
\none[3] & \none & \cdot & \cdot \cr
\none[4] & \none & \cdot & \none \cr
\none[5] & \none & \none & \none 
\end{ytableau}}
$$
\end{exa}

\subsection{\texorpdfstring{$K$}{TEXT}-Kohnert diagrams}
We recall a combinatorial formula for Lascoux polynomials. 
To simplify our description, 
we introduce the following definition.
\begin{defn}
A \definition{labeled diagram} 
is a diagram where each cell can be labeled 
by a symbol.
The \definition{underlying diagram} of a labeled diagram
is the diagram obtained by ignoring all labels. 
The weight of a labeled diagram $D$, denoted as $\wt(D)$,
is just the weight of its underlying diagram.
\end{defn}

Then a \definition{ghost diagram} is a labeled diagram 
where cells can be labeled by $\X$.
We call cells labeled by $\X$ as ``\definition{ghosts}''. 
For a ghost diagram $D$, its \definition{excess}, denoted as $\ex(D)$,
is the number of ghosts in $D$.
Next, we define a move on ghost diagrams.

\begin{defn}[\cite{RossY}]
A \definition{$K$-Kohnert move} is defined on a ghost diagram $D$. 

We pick a cell $(r,c)$ and move it up, 
subject to the following requirements.
\begin{itemize}
\item The cell $(r,c)$ must be the rightmost cell in row $r$.
\item The cell $(r,c)$ is not a ghost.
\item The cell $(r,c)$ is moved to the lowest empty spot above it.
\item The cell $(r,c)$ may jump over other cells but cannot jump over any ghosts.
\end{itemize}
After the move, 
we may or may not leave a ghost at $(r,c)$. 
When we leave a ghost, 
we refer this move as a \definition{ghost move}.
\end{defn}

For a weak composition $\alpha$,
a ghost diagram is called a \definition{$K$-Kohnert diagram}
of $\alpha$ if it can be obtained from $D(\alpha)$
by $K$-Kohnert moves. 
Let $\KKD(\alpha)$ be the set of all $K$-Kohnert diagrams
of $\alpha$.
As proved in~\cite{PY}, 
$K$-Kohnert diagrams give a formula for Lascoux polynomials. 
This rule was first conjectured by Ross and Yong~\cite{RossY}.
Notice that our convention is different from~\cite{PY}:
row $1$ is the top most row in this paper while it is the 
bottom most row in~\cite{PY}. 

\begin{ourthm}[\cite{PY}]\label{thm: k.kohnert}
Let $\alpha$ be a weak composition.
Then we have
$$
\fLb_\alpha = \sum_{D \in \KKD(\alpha)}x^{\wt(D)} \beta^{\ex(D)}.
$$
\end{ourthm}

\begin{exa} \label{eg_021} Let $\alpha = (0,2,1)$, then 
$KKD(\alpha)$ consists of the following: 
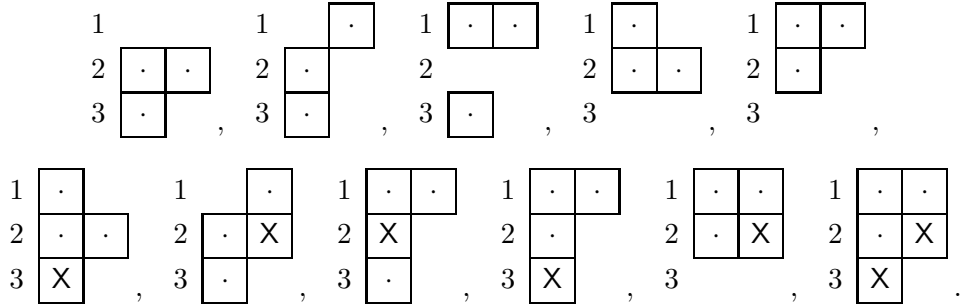
\begin{align*}
\raisebox{1.1cm}{
\begin{ytableau}
\none[1] \cr
\none[2] & \cdot & \cdot \cr
\none[3] & \cdot 
\end{ytableau}}\:,
\raisebox{1.1cm}{
\begin{ytableau}
\none[1] &\none & \cdot \cr
\none[2] & \cdot \cr
\none[3] & \cdot 
\end{ytableau}}\:,
\raisebox{1.1cm}{
\begin{ytableau}
\none[1] & \cdot & \cdot \cr
\none[2]  \cr
\none[3] & \cdot 
\end{ytableau}}\:,
\raisebox{1.1cm}{
\begin{ytableau}
\none[1] & \cdot \cr
\none[2] & \cdot & \cdot \cr
\none[3] 
\end{ytableau}}\:,
\raisebox{1.1cm}{
\begin{ytableau}
\none[1] & \cdot & \cdot \cr
\none[2] & \cdot  \cr
\none[3] 
\end{ytableau}}\:,
\end{align*}
\begin{align*}
\raisebox{1.1cm}{
\begin{ytableau}
\none[1] & \cdot \cr
\none[2] & \cdot & \cdot \cr
\none[3] & \X 
\end{ytableau}}\:,   
\raisebox{1.1cm}{
\begin{ytableau}
\none[1] &\none & \cdot \cr
\none[2] & \cdot & \X \cr
\none[3] & \cdot 
\end{ytableau}}\:,
\raisebox{1.1cm}{
\begin{ytableau}
\none[1] & \cdot & \cdot \cr
\none[2] & \X \cr
\none[3] & \cdot 
\end{ytableau}}\:,
\raisebox{1.1cm}{
\begin{ytableau}
\none[1] & \cdot & \cdot \cr
\none[2] & \cdot \cr
\none[3] & \X
\end{ytableau}}\:,
\raisebox{1.1cm}{
\begin{ytableau}
\none[1] & \cdot & \cdot \cr
\none[2] & \cdot & \X \cr
\none[3] 
\end{ytableau}}\:,
\raisebox{1.1cm}{
\begin{ytableau}
\none[1] & \cdot & \cdot \cr
\none[2] & \cdot & \X\cr
\none[3] & \X
\end{ytableau}}\:.   
\end{align*}

By the rule above, we have
\begin{equation*}
\begin{split}
\fLb_{\alpha} = & \: x_2^2x_3 + x_1x_2x_3 + x_1^2x_3 
+ x_1x_2^2 + x_1^2x_2\\
+ & \: \beta(x_1x_2^2x_3 + x_1x_2^2x_3 + x_1^2x_2x_3 + x_1^2x_2x_3 + x_1^2x_2^2) + \beta^2 x_1^2x_2^2x_3.
\end{split}
\end{equation*}
\end{exa}

\section{Snow diagrams}
\label{S: Cloud and snow}
We associate each diagram with a labeled diagram
called the \definition{snow diagram} which allows us to define two statistics
on diagrams. 
For each diagram $D$, 
we describe the following algorithm that outputs $\snow(D)$.
Cells in $\snow(D)$ can be labeled by $\bullet$ or $\snowflake$.

\begin{enumerate}
\item[-] Iterate 
through rows of $D$ from bottom to top. 
\item[-] In each row $r$ of $D$, find the rightmost cell $(r,c)$ 
with no $\bullet$ in column $c$.
If such an $(r,c)$ exists,
label it by $\bullet$
and put a cell labeled by $\snowflake$ in $(r', c)$ for $r' \in [r-1]$ and $(r', c) \notin D$.
\end{enumerate}

We call cells labeled by $\bullet$ \definition{dark clouds}
and cells labeled by $\snowflake$  \definition{snowflakes}.

\begin{exa}\label{exa.snow}
The following is a diagram together with its snow diagram. 
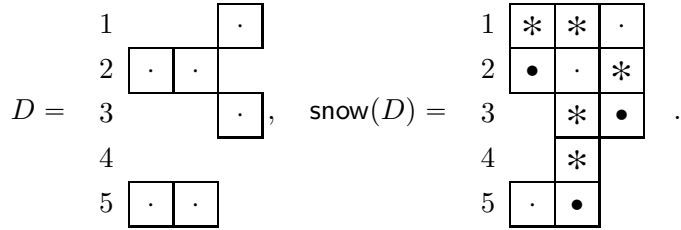
\begin{align*}
D = 
\raisebox{0.95cm}{
\begin{ytableau}
\none[1] & \none& \none & \cdot \cr
\none[2] & \cdot& \cdot & \none \cr
\none[3] & \none& \none & \cdot \cr
\none[4] & \none& \none & \none \cr
\none[5] & \cdot& \cdot \cr
\end{ytableau}}\,,\quad
\snow(D) = 
\raisebox{0.95cm}{
\begin{ytableau}
\none[1] & \snowflake& \snowflake & \cdot \cr
\none[2] & \bullet& \cdot & \snowflake \cr
\none[3] & \none& \snowflake & \bullet \cr
\none[4] & \none& \snowflake & \none \cr
\none[5] & \cdot& \bullet \cr
\end{ytableau}}\quad.
\end{align*}
\end{exa}

The positions of dark clouds will be important,
so we make the following definition.

\begin{defn}
The \definition{dark cloud diagram} of a diagram $D$,
$\dark(D)$, 
is the set of cells $(r,c)$ that are dark clouds in $\snow(D)$.
\end{defn}

\begin{exa}
In Example~\ref{exa.snow},
$\dark(D) = \{(2,1), (3,3), (5,2) \}$.
\end{exa}

A diagram is a 
\definition{non-attacking rook diagram}
if it has at most one cell in each row or column.
Let $\Rook_+$ be the family
of all non-attacking rook diagrams.

\begin{rem}
\label{R: Basic of dark}
We make the following observations about $\dark(D)$.
\begin{itemize}
\item By construction, $\dark(D) \in \Rook_+$.
\item Take $(r,c) \in D$.
If there are no $r' > r$ with $(r', c) \in \dark(D)$ 
and there are no $c' > c$ with $(r, c') \in \dark(D)$,
then $(r,c) \in \dark(D)$.
\end{itemize}
\end{rem}

Finally, we associate two statistics
to each diagram via its snow diagram.

\begin{defn}\label{defn.raj}
Let $D$ be a diagram.
The \definition{rajcode} of $D$,
$\rajcode(D)$, 
is the weak composition $\wt(\snow(D))$.
Let $\raj(D)$ denote $|\rajcode(D)|$,
the total number of cells in $\snow(D)$.
\end{defn}

\begin{exa}
Continuing with Example~\ref{exa.snow}, we have $\rajcode(D) = (3,3,2,1,2)$ and $\raj(D) = 11$.
\end{exa}

\begin{rem}
Recall that Pechenik, Speyer and Weigandt~\cite{PSW} 
define the statistics $\rajcode(\cdot)$ and $\raj(\cdot)$ on permutations
using increasing subsequences. 
We show that our $\rajcode$ and $\raj$ on Rothe diagrams 
agree with their definitions in Theorem~\ref{T: equivalence of rajcode on perm}.
Therefore, our construction on Rothe diagrams is a diagrammatic way
to compute the leading monomial and degree of $\topGro_w$.
In addition, we notice that positions of dark clouds in $\snow(RD(w))$
are connected to the Schensted insertion and Viennot's geometric construction.
These connections are explored in \S\ref{S: permutations}.
\end{rem}

\section{Proof of Theorem~\ref{T: Top Las}}
\label{S: weak composition}

To prove Theorem~\ref{T: Top Las},
we study top Lascoux polynomials via snow diagrams of key diagrams.
With a slight abuse of notation, we
define $\rajcode(\alpha) := \rajcode(D(\alpha))$, $\raj(\alpha) := \raj(D(\alpha))$
and $\dark(\alpha) = \dark(D(\alpha))$
for $\alpha \in C_+$.
We start by introducing some definitions.

\begin{defn}
A weak composition $\alpha$ is called \definition{snowy}
if its positive entries are all distinct. 
\end{defn}

Our main goal in this section is to establish 
Theorem~\ref{T: Top Las}:

\topLasThm*

This task is broken into four major lemmas
established in the following four subsections. 
In Subsection~\ref{SS: Top Las: Exist}, 
we use $K$-Kohnert diagrams to establish the first major lemma:

\begin{restatable}{lem}{topTermlem}
\label{L: Top Las 1}
The polynomial $\fLb_\alpha$ has the term $x^{\rajcode(\alpha)} \beta^{\raj(\alpha) - |\alpha|}$.
\end{restatable}

Lemma~\ref{L: Top Las 1} proves
$\topLas_\alpha$ has degree at least $\raj(\alpha)$.
To show $\topLas_\alpha$ indeed has degree $\raj(\alpha)$,
we need the following equivalence relation on weak compositions. 

\begin{defn}
Let $\alpha$ and $\gamma$ be two weak compositions. 
We say $\alpha$ is \definition{rajcode equivalent} to $\gamma$, 
denoted as $\alpha \sim \gamma$, 
if $\rajcode(\alpha) = \rajcode(\gamma)$. 
\end{defn}

\begin{exa}
\label{E: sim of weak comp}
Let $\alpha = (2,0,4,3,1)$ and $\gamma = (3,1,4,3,1)$.
Then we have:
\begin{align*}
D(\alpha) = 
\raisebox{0.95cm}{
\begin{ytableau}
\none[1] & \cdot& \cdot & \none & \none \cr
\none[2] & \none& \none & \none & \none \cr
\none[3] & \cdot& \cdot & \cdot & \cdot \cr
\none[4] & \cdot& \cdot & \cdot & \none \cr
\none[5] & \cdot& \none & \none & \none \cr
\end{ytableau}}\quad,
\quad \quad
\snow(D(\alpha)) = 
\raisebox{0.95cm}{
\begin{ytableau}
\none[1] & \cdot& \bullet & \snowflake & \snowflake \cr
\none[2] & \snowflake &  & \snowflake & \snowflake \cr
\none[3] & \cdot& \cdot & \cdot & \bullet \cr
\none[4] & \cdot& \cdot & \bullet & \none \cr
\none[5] & \bullet& \none & \none & \none \cr
\end{ytableau}}\quad,
\\
\end{align*}

\begin{align*}
D(\gamma) = 
\raisebox{0.95cm}{
\begin{ytableau}
\none[1] & \cdot& \cdot & \cdot & \none \cr
\none[2] & \cdot& \none & \none & \none \cr
\none[3] & \cdot& \cdot & \cdot & \cdot \cr
\none[4] & \cdot& \cdot & \cdot & \none \cr
\none[5] & \cdot& \none & \none & \none \cr
\end{ytableau}}\quad,
\quad \quad
\snow(D(\gamma)) = 
\raisebox{0.95cm}{
\begin{ytableau}
\none[1] & \cdot& \bullet & \cdot & \snowflake \cr
\none[2] & \cdot &  & \snowflake & \snowflake \cr
\none[3] & \cdot& \cdot & \cdot & \bullet \cr
\none[4] & \cdot& \cdot & \bullet & \none \cr
\none[5] & \bullet& \none & \none & \none \cr
\end{ytableau}}\quad.
\end{align*}

Be aware that the cell $(2,2)$ is not in
$\snow(D(\alpha))$ or $\snow(D(\gamma))$.
Observe that $\rajcode(\alpha) = (4,3,4,3,1) = \rajcode(\gamma)$,
so $\alpha \sim \gamma$.
\end{exa}

In Subsection~\ref{SS: equivalence},
we study this equivalence relation.
We show that snowy weak compositions form a complete set of representatives:

\begin{restatable}{lemma}{LasB}
\label{L: snowy weak composition in equivalence class}
For each equivalence class of $\sim$, 
there is a unique $\alpha$ such that $\alpha$ is snowy. 
Moreover, if $\gamma \sim \alpha$ and $\alpha$ is snowy, 
then $\gamma_r \geqslant \alpha_r$ for all $r$.
In other words, a snowy weak composition is the
unique entry-wise minimum in each equivalence class. 
\end{restatable}

In Subsection~\ref{SS: snowy}, 
we focus on $\topLas_\alpha$ for snowy $\alpha$ and give a recursive description of $\topLas_\alpha$,
which leads to the third major lemma.

\begin{restatable}{lemma}{LasC}
\label{L: snowy leading}
If $\alpha$ is snowy, 
then $x^{\rajcode(\alpha)}$ is the leading monomial of 
$\topLas_\alpha$ with coefficient 1.
\end{restatable}

Finally, we devote the Subsection~\ref{SS: last major lemma} to proving the last
major lemma:

\begin{restatable}{lemma}{LasD}
\label{L: Top Las similar}
If $\alpha \sim \gamma$, 
then $\topLas_\alpha = c\topLas_\gamma$ for some $c \neq 0$.
\end{restatable}

Once we have these four major lemmas, 
we can easily check Theorem~\ref{T: Top Las}.

\begin{proof}
First, statement (c) follows from Lemma~\ref{L: snowy weak composition in equivalence class} and Lemma~\ref{L: snowy leading}. 

Given a weak composition $\alpha$. Let $\beta$ be the unique snowy weak composition such that $\alpha \sim \beta$. Statement (a) follows from  Lemma~\ref{L: snowy leading} and Lemma~\ref{L: Top Las similar}.

For statement (b), the backward direction is just Lemma~\ref{L: Top Las similar}.
For the forward direction, 
if $\topLas_\alpha$ is a scalar multiple of $\topLas_\gamma$,
then they have the same leading monomial.
By statement (a), we have $\rajcode(\alpha) = \rajcode(\gamma)$.
\end{proof}

\subsection{Proof of Lemma~\ref{L: Top Las 1}}
\label{SS: Top Las: Exist}

We show the monomial $ x^{\rajcode(\alpha)}\beta^{\raj(\alpha) - |\alpha|}$ 
exists in $\fLb_\alpha$.
We give an algorithm whose output is a $K$-Kohnert diagram for $\alpha$, which has the same underlying diagram as $\snow(D(\alpha))$. First, observe that $\snow(D(\alpha))$ contains no dark clouds if and only if $\alpha$ contains only zero entries. In this case, $\topLas_\alpha = 1$ and $\rajcode(\alpha)$ only has zero entries. Our claim is immediate. 
In the rest of this subsection, we assume $\alpha$ is a weak composition
with at least one positive entry, and thus $\snow(D(\alpha))$ has at least one dark cloud.
To describe the algorithm, 
we introduce two useful moves on ghost diagrams.
\begin{defn}
Let $D$ be a ghost diagram.
Let $(r,c)$ be a non-ghost cell in $D$
and let $(r', c)$ be the highest empty space in column $c$.
If $r' < r$,
let $UP_{(r,c)}(D)$ be the diagram we get after 
moving $(r,c)$ to $(r',c)$. 
Let $UP^G_{(r,c)}(D)$ be the diagram we get after 
moving $(r,c)$ to $(r',c)$
and putting a ghost on $(r,c)$ and all empty spaces between $(r,c)$ and $(r',c)$.
If $r' > r$, define 
$UP^G_{(r,c)}(D) = UP_{(r,c)}(D) = D$.
\end{defn}

\begin{rem}
\label{R: K-Kohnert move}
Assume $UP_{(r,c)}$ or $UP^G_{(r,c)}$ moves a cell to $(r',c)$.
Then this move can be achieved by a sequence of $K$-Kohnert moves
if both of the following conditions hold for each $r' < j \leq r$:
\begin{itemize}
\item If $(j, c)\notin D$,
then $D$ has no cell to the right of column $c$ in row $j$.
\item If $(j, c)\in D$, then it is not a ghost cell.
\end{itemize}
\end{rem}

Now we can describe the algorithm.
Let $D^{0} = D(\alpha)$. Recall by Remark~\ref{R: Basic of dark}, there is at most one dark cloud in each column of $\snow(D(\alpha))$. 
We can label all the dark clouds as $(r_1,c_1),\dots, (r_{m},c_{m})$ where $c_1<c_2\dots <c_m$ for some $m\geqslant 1$. We iterate $i$ from $1$ to $m$.
At iteration $i$, compute

\begin{align}
\label{EQ: UP}
D^i = UP^G_{(r_i,c_i)} \circ UP_{(r_i,c_i + 1)} \cdots \circ UP_{(r_i,\alpha_{r_i})} (D^{i-1})\,.    
\end{align}

\begin{exa}
Consider $\alpha = (1,3,4,0,4,3)$, we compute its snow diagram
and we have the dark clouds at $(2,1),(3,2),(6,3),(5,4)$. We compute $D^4$ according to the above algorithm. 
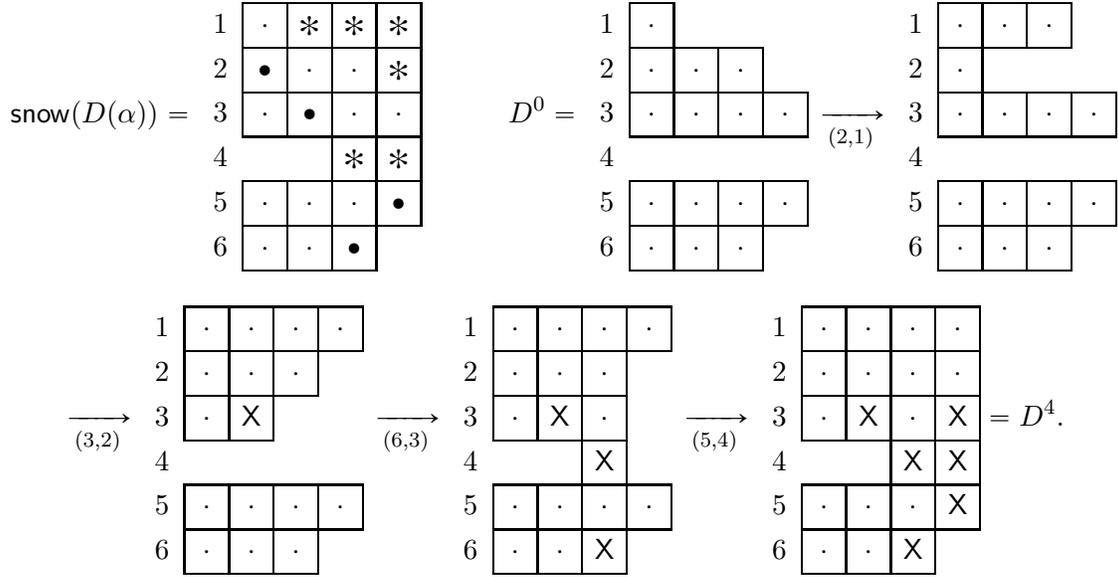
\[
\snow(D(\alpha)) = 
\raisebox{1cm}{\begin{ytableau}
\none[1] & \cdot & \snowflake & \snowflake & \snowflake \cr
\none[2] & \bullet & \cdot & \cdot & \snowflake \cr
\none[3] & \cdot & \bullet & \cdot & \cdot \cr
\none[4] & \none & \none & \snowflake & \snowflake\cr
\none[5] & \cdot& \cdot & \cdot & \bullet \cr
\none[6] & \cdot& \cdot & \bullet \cr
\end{ytableau}}\quad\quad\quad
D^0 = 
\raisebox{1cm}{\begin{ytableau}
\none[1] & \cdot \cr
\none[2] & \cdot & \cdot & \cdot \cr
\none[3] & \cdot & \cdot & \cdot & \cdot \cr
\none[4] & \none \cr
\none[5] & \cdot & \cdot & \cdot & \cdot \cr
\none[6] & \cdot & \cdot & \cdot \cr
\end{ytableau}}\, \xrightarrow[(2,1)]{}
\raisebox{1cm}{\begin{ytableau}
\none[1] & \cdot & \cdot & \cdot \cr
\none[2] & \cdot \cr
\none[3] & \cdot & \cdot & \cdot & \cdot \cr
\none[4] & \none \cr
\none[5] & \cdot & \cdot & \cdot & \cdot \cr
\none[6] & \cdot & \cdot & \cdot \cr
\end{ytableau}}\,
\]
\[
 \xrightarrow[(3,2)]{}\raisebox{1cm}{\begin{ytableau}
\none[1] & \cdot & \cdot & \cdot & \cdot \cr
\none[2] & \cdot & \cdot & \cdot \cr
\none[3] & \cdot & \X  \cr
\none[4] & \none \cr
\none[5] & \cdot & \cdot & \cdot & \cdot \cr
\none[6] & \cdot & \cdot & \cdot \cr
\end{ytableau}}\,\xrightarrow[(6,3)]{}
\raisebox{1cm}{\begin{ytableau}
\none[1] & \cdot & \cdot & \cdot & \cdot \cr
\none[2] & \cdot & \cdot & \cdot \cr
\none[3] & \cdot & \X &  \cdot \cr
\none[4] & \none & \none & \X \cr
\none[5] & \cdot & \cdot & \cdot & \cdot \cr
\none[6] & \cdot & \cdot & \X \cr
\end{ytableau}}\,\xrightarrow[(5,4)]{}
\raisebox{1cm}{\begin{ytableau}
\none[1] & \cdot & \cdot & \cdot & \cdot \cr
\none[2] & \cdot & \cdot & \cdot & \cdot \cr
\none[3] & \cdot & \X    &  \cdot    &  \X\cr
\none[4] & \none & \none & \X    & \X \cr
\none[5] & \cdot & \cdot & \cdot & \X \cr
\none[6] & \cdot & \cdot & \X \cr
\end{ytableau}} = D^{4}.
\]
\end{exa}

We observe that in the previous example, 
$D^4$ has the same underlying diagram as $\snow(D(\alpha))$.
This is true in general.

\begin{lem}
The labeled diagram $D^m$ defined by~(\ref{EQ: UP})
has the same underlying diagram as $\snow(D(\alpha))$.
\end{lem}
\begin{proof}
For a number $c$, 
we compare the column $c$ 
of $\snow(D(\alpha))$ and $D^m$.
If column $c$ of $\snow(D(\alpha))$ has no dark cloud,
then it is the same as column $c$ of $D(\alpha)$.
In this case, the algorithm will not move any cells 
in column $c$.
Thus, $D^m$ and $D(\alpha)$ also agree in column $c$.

Now suppose $\snow(D(\alpha))$ has a dark cloud
in column $c$, say at row $r$.
In the underlying diagram of $\snow(D(\alpha))$,
column $c$ is obtained from column $c$ of $D(\alpha)$ 
by filling all empty spaces above row $r$.  
On the other hand, consider what the algorithm does on column $c$.
It first might move cells above row $r$
and then it fills all empty spaces weakly above row $r$.
Thus, column $c$ in $D^m$ is the same as 
column $c$ of $\snow(D(\alpha))$ after ignoring the labels.  
\end{proof}

Next, we want to show $D^m$ produced by the algorithm
is in $\KKD(\alpha)$.
We just need to check each $UP_{(r,j)}$ and $UP^G_{(r,c)}$ in each iteration is a sequence of $K$-Kohnert moves. To that end, we first make the following observation about the diagram $D^i$.

\begin{lemma}
Let $c_0 = 0$. In $D^i$,
if a cell is strictly to the right of column $c_i$,
then there is a cell immediately on its left.
In other words, 
the diagram $D^i$ is left-justified if we ignore the first $c_i$ columns.
\end{lemma}

\begin{proof}
Prove by induction on $i$.
The lemma holds for $D^0$, which is left-justified.

Assume $D^{i-1}$ is left-justified if we ignore the first $c_{i-1}$ columns,
for some $i \geq 1$.
Consider an arbitrary cell $(r,c)$ in $D^i$
with $c > c_i$.
We show $(r, c-1)$ is in $D^i$ by considering two possibilities.
\begin{itemize}
\item[-] The cell $(r,c)$ is not in $D^{i-1}$.
Then during iteration $i$, a cell is moved to $(r,c)$,
which is the highest blank in column $c$ of $D^{i-1}$.
By our inductive hypothesis and $c-1 > c_{i-1}$, 
the highest blank in column $c-1$ of $D^{i-1}$
is weakly lower than row $r$.
Thus, $(r,c-1)$ is in $D^i$.
\item[-] Otherwise, $(r,c)$ is in $D^{i-1}$.
By our inductive hypothesis, 
$(r, c-1)$ is in $D^{i-1}$.
If $r \neq r_i$,
then we know that no cell from row $r$ is moved 
during iteration $i$.
Thus, $(r,c-1)$ is still in $D^i$.
If $r = r_i$,
then there are no empty spaces above $(r,c)$
in $D^{i-1}$.
By our inductive hypothesis, 
there is no empty spaces above $(r,c-1)$,
so $(r,c-1)$ is still in $D^i$. \qedhere
\end{itemize}
\end{proof}

The above lemma shows that the diagram $D^i$ is left-justified if we ignore the first $c_i$ columns.
We will use this property to show that $D^m$ 
is in $\KKD(\alpha)$.
\begin{prop}\label{P: Top Las 1}
The above algorithm can be achieved by $K$-Kohnert moves,
so $D^m \in \KKD(\alpha)$.
\end{prop}

\begin{proof}
We focus on one iteration of the algorithm,
say iteration $i$.
We check the operators 
in~(\ref{EQ: UP}) can be achieved by $K$-Kohnert moves. 
We ignore all cells to the left of the column $c_i$ in $D^{i-1}$.
By the previous Lemma, this part of the diagram 
is left-justified.
The highest empty spaces in columns $c_i, \cdots, \alpha_{r_i}$
are going weakly up from left to right.
Moreover, the condition in Remark~\ref{R: K-Kohnert move} 
holds for all $(r_i, c_i), \cdots, (r_i, \alpha_{r_i})$.

Now $UP_{(r_i, \alpha_{r_i})}$
can be achieved by $K$-Kohnert moves.
After that, the conditions in Remark~\ref{R: K-Kohnert move} 
hold at each step for $(r_i, \alpha_{r_i} - 1), \dots, (r_i, c_i)$.
Following this logic, this iteration can be achieved by $K$-Kohnert moves. 
\end{proof}

Using Theorem~\ref{thm: k.kohnert}:

\topTermlem*

\subsection{Proof of Lemma~\ref{L: snowy weak composition in equivalence class}}
\label{SS: equivalence}

First, notice that we can recover the underlying diagram of $\snow(D(\alpha))$
from $\dark(\alpha)$.

\begin{lemma}
\label{L: dark recovers snow}
Let $\alpha$ be a weak composition. 
The underlying diagram of $\snow(D(\alpha))$
is:
\begin{align}
\label{Eq: snow recovered by dark}
\bigcup_{(r,c) \in \dark(\alpha)} ([r] \times \{c\}) \cup (\{r\} \times [c]).
\end{align}
\end{lemma}
\begin{proof}
First, we show that the elements of the set~(\ref{Eq: snow recovered by dark})
are cells in $\snow(D(\alpha))$.
Take $(r,c) \in \dark(\alpha)$.
We know $(r,c) \in D(\alpha)$.
Since $D(\alpha)$ is left-justified, 
$\{r\} \times [c] \subseteq D(\alpha)$.
Thus, these cells are in $\snow(D(\alpha))$.
By the construction of $\snow(D(\alpha))$,
the cells in $[r] \times \{c\}$ are also 
in $\snow(D(\alpha))$.

Now suppose there is a cell $(r, c)$ 
in $\snow(D(\alpha))$
that is not in the 
set~(\ref{Eq: snow recovered by dark}).
Then there is no $r' > r$ with $(r', c) \in \dark(D)$,
which implies $(r,c)$ is not a snowflake 
in $\snow(D(\alpha))$.
Thus, $(r,c) \in D(\alpha)$.
Also, there is no $c' > c$ with $(r, c') \in \dark(D)$.
By Remark~\ref{R: Basic of dark},
$(r,c) \in \dark(D)$.
Thus, $(r, c)$ is in the 
set~(\ref{Eq: snow recovered by dark}),
which is a contradiction.
\end{proof}

Furthermore, 
we can recover $\dark(\alpha)$ from $\rajcode(\alpha)$.

\begin{lemma}
\label{L: same rajcode => same dark}
Let $\alpha, \gamma$ be weak compositions.
If $\rajcode(\alpha) = \rajcode(\gamma)$,
then $\dark(\alpha) = \dark(\gamma)$.
\end{lemma}
\begin{proof}
We prove the two diagrams $\dark(\alpha)$ 
and $\dark(\gamma)$ agree on each row $r$,
by a reverse induction on $r$.
The base case is immediate.
Suppose $r$ is large enough
such that $\alpha_i = \gamma_i = 0$
if $i > r$.
Then $\dark(\alpha)$ and $\dark(\gamma)$
clearly agree on row $r$ and underneath.

Next, we show that the value $\rajcode(\alpha)_r$ and cells in $\dark(\alpha)$ under row $r$ 
determines whether $\dark(\alpha)$ has a cell on row $r$.
Moreover, if such a cell exists, its column index is also determined. 

Let $r\geqslant 1$.
Define
$$B_r :=\{c: \textrm{There are no dark clouds under } (r,c) \textrm{ in } \snow(D(\alpha) \}.$$

The complement of $B_r$ is $\overline{B_r} := \mathbb{Z}_{>0} - B_r = \{c: (r',c) \in \dark(\alpha) \textrm{ for some } r' > r \}$.
For $c \in \overline{B_r}$,
$(r,c)$ of $\snow(D(\alpha))$ is a snowflake or an unlabeled cell.
If there is no dark cloud on row $r$ of $\snow(D(\alpha))$,
$\rajcode(\alpha)_r = |\overline{B_r}|$.
Otherwise, we assume the dark cloud is at $(r,c)$ for some $c \in B_r$.
Then row $r$ of $\snow(D(\alpha))$ has cells on 
$(r,c')$ for $c' \in \overline{B_r}$ or $c' \leq c$.
Suppose $c$ is the $i$\textsuperscript{th} smallest number in $B_r$.
We have $\rajcode(\alpha)_r = i + |\overline{B_r}|$.

Consequently, $\rajcode(\alpha)_r$ and $\dark(\alpha)$ under row $r$
uniquely determines row $r$ of $\dark(\alpha)$.
If we assume $\dark(\alpha)$ and $\dark(\gamma)$
agree underneath row $r$ as our inductive hypothesis,
then they also agree on row $r$ since $\rajcode(\alpha)_r = \rajcode(\gamma)_r$.
The induction is finished. 
\end{proof}

Now we have two equivalent ways of describing rajcode equivalence.
\begin{prop}
\label{P: equivalence}
Let $\alpha$ and $\gamma$ be two weak compositions. 
The following are equivalent:
\begin{enumerate}
\item $\alpha \sim \gamma$;
\item $\dark(\alpha) = \dark(\gamma)$.
\item The underlying diagrams of $\snow(D(\alpha))$ and  $\snow(D(\gamma))$
are the same;
\end{enumerate}
\end{prop}
\begin{proof}
By Lemma~\ref{L: same rajcode => same dark}, (1) implies (2).
By Lemma~\ref{L: dark recovers snow}, (2) implies (3).
Clearly, (3) implies (1).
\end{proof}

Our next goal is to find representatives
of rajcode equivalence classes. 
At the end of this subsection, 
we will see snowy weak compositions form
a complete set of representatives. 
To understand snowy weak compositions, 
we start with the following observation. 

\begin{rem}
\label{R: snowy}
For a weak composition $\alpha$,
the following are equivalent:
\begin{itemize}
\item $\alpha$ is snowy.
\item The rightmost cell in each row of $D(\alpha)$
are in different columns. 
\item The rightmost cell in each row of $D(\alpha)$
is a dark cloud in $\snow(D(\alpha))$. 
\end{itemize}
\end{rem}

One advantage of working with snowy weak compositions
is that 
we can tell their 
$\rajcode(\cdot)$ and $\raj(\cdot)$ easily:
\begin{lemma}
\label{L: dark, rajcode and raj of snowy weak composition}
Let $\alpha$ be a snowy weak composition.
Then the following statements hold.
\begin{enumerate}
\item $\dark(\alpha) = \{(r,\alpha_r): \alpha_r > 0\}$,\label{snowy.weak.1}
\item $\rajcode(\alpha)_r = \alpha_r + |\{r' > r: \alpha_r < \alpha_r'\}|$, and \label{snowy.weak.2}
\item $\raj(\alpha) = \sum_r (\alpha_r + |\{(r, r'): \alpha_r < \alpha_r', r < r'\}|) = |\alpha| + |\{(r, r'): r < r', \alpha_r < \alpha_{r'}\}|$.\label{snowy.weak.3}
\end{enumerate}
\end{lemma}
\begin{proof}
(1) follows from Remark~\ref{R: snowy}.
(2) follows from (1) and Lemma~\ref{L: dark recovers snow}, and 
(3) immediately follows from (2). 
\end{proof}

As a consequence, we have the following rule which tells us 
how $\rajcode(s_i\alpha)$ differs from 
$\rajcode(\alpha)$ when $\alpha$ is snowy. 

\begin{cor}
\label{C: snowy changes}
Let $\alpha$ be a snowy weak composition and consider $i$ with $\alpha_i > \alpha_{i+1}$.
Then $\rajcode(s_i\alpha) = s_i\rajcode(\alpha) + e_i$,
where $e_i$ is the weak composition with
1 on its $i$\textsuperscript{th} entry and 0 elsewhere.
\end{cor}

The second advantage of working with snowy weak compositions is that
they are in bijection with $\Rook_+$.

\begin{lem}
\label{L: snowy biject rook}
The map $\dark(\cdot)$ is a bijection
from $\{ \alpha \in C_+: \alpha \textrm{ is snowy}\}$
to $\Rook_+$.
Its inverse $\dark^{-1}(\cdot)$ 
is given by $\dark^{-1}(R) = \alpha$
where \begin{align*}
\alpha_r =
\begin{cases}
0 &\text{if row $r$ of $R$ is empty;} \\
c &\text{if $(r,c) \in R$.} 
\end{cases}
\end{align*}
\end{lem}
\begin{proof}
Follows from Remark \ref{R: snowy}.
\end{proof}

We are ready to show that they are representatives of all equivalence classes.

\LasB* 
\begin{proof}
Let $\gamma$ be an arbitrary weak composition. 
First, we construct a snowy $\alpha$ 
such that $\alpha \sim \gamma$.
We know $\dark(\gamma) \in \Rook_+$.
We send it to a snowy $\alpha$
using the map in Lemma~\ref{L: snowy biject rook}.
Then $\dark(\alpha) = \dark(\gamma)$.
By Proposition~\ref{P: equivalence},
$\alpha \sim \gamma$.

Next, take a positive integer $r$.
If $\alpha_r = 0$, then $\gamma_r \geqslant \alpha_r$ trivially.
Otherwise, we know $(r, \alpha_r) \in \dark(\alpha) = \dark(\gamma)$.
Thus, $\gamma_r \geqslant \alpha_r$.

Finally, we establish the uniqueness of this snowy $\alpha$.
Assume $\alpha'$ is a snowy weak composition
such that $\alpha' \sim \gamma$.
Then $\alpha'_r \geqslant \alpha_r$ and $\alpha_r \geqslant \alpha'_r$
for all $r \in \Z_{>0}$, so $\alpha = \alpha'$.
\end{proof}

A snowy weak composition has more snowflakes in its snow diagram 
than any others in its equivalence class; hence the name.
Say $\alpha \sim \gamma$ and
$\alpha$ is snowy while $\gamma$ is not. 
By Lemma~\ref{L: snowy weak composition in equivalence class}, $|\alpha| <  |\gamma|$.
On the other hand, 
the number of snowflakes in $\snow(D(\alpha))$ 
(resp. $\snow(D(\gamma))$)
is $\raj(\alpha) - |\alpha|$ (resp. $\raj(\gamma) - |\gamma|$). 
Since $\raj(\alpha) = \raj(\gamma)$, $\snow(D(\alpha))$
has more snowflakes than $\snow(D(\gamma))$.

\subsection{Proof of Lemma~\ref{L: snowy leading}}
\label{SS: snowy}

By Lemma~\ref{L: Top Las 1},
$\topLas_\alpha$ has degree at least $\raj(\alpha)$.
Next, we can show the degree of $\topLas_\alpha$
equals to $\raj(\alpha)$ when $\alpha$ is snowy. 

\begin{lemma}
\label{L: Snowy degree}
Let $\alpha$ be a snowy weak composition.
The $\beta$-degree of $\fLb_\alpha$ 
is $\raj(\alpha) - |\alpha|$,
so the degree of $\topLas_\alpha$ is $\raj(\alpha)$.
\end{lemma}
\begin{proof}
We prove the result by induction on 
\begin{align*}
\ell(\alpha) := |\{(i,j) \mid\, \alpha_i < \alpha_j \text{ and } i < j\}|.
\end{align*}
For the base case, if $\ell(\alpha) = 0$,
then $\alpha$ is weakly decreasing.
The polynomial $\fLb_\alpha$ is an monomial with $\beta$-degree $0$.
Correspondingly, $\raj(\alpha) = |\alpha|$.

Now if $\ell(\alpha) > 0$,
we can find $i$ with $\alpha_i < \alpha_{i+1}$.
By Corollary~\ref{C: snowy changes},
$\raj(s_i \alpha) = \raj(\alpha) - 1$.
Notice that $\ell(s_i \alpha) = \ell(\alpha) - 1$.
By our inductive hypothesis, 
the $\beta$-degree of $\fLb_{s_i \alpha}$
is $\raj(s_i \alpha) - |\alpha| = \raj(\alpha) - 1 - |\alpha|$.
By the recursive definition of Lascoux polynomials,
$$
\fLb_\alpha 
= \pi_i(\fLb_{s_i\alpha}) + \beta \pi_i(x_{i+1}\fLb_{s_i\alpha}).
$$
The $\beta$-degree in $\fLb_\alpha$
is at most $\raj(\alpha) - |\alpha|$.
Lemma~\ref{L: Top Las 1} implies the 
$\beta$-degree of $\fLb_\alpha$ is at least 
$\raj(\alpha) - |\alpha|$,
so the inductive step is finished.
\end{proof}

Combine with Lemma~\ref{L: dark, rajcode and raj of snowy weak composition},
we have: 
\begin{cor}
Let $\alpha$ be a snowy weak composition.
The degree of $\topLas_\alpha$
is $|\alpha| + |\{(r, r'): r < r', \alpha_r < \alpha_{r'}\}|$.
\end{cor}

Now we can describe $\topLas_\alpha$ for snowy 
$\alpha$ recursively.
\begin{lemma}\label{L: recursive top Las}
Let $\alpha$ be a snowy weak composition. 
Then
\begin{align}
\topLas_\alpha = \begin{cases}
x^\alpha & \text{if $\alpha_1 \geqslant \alpha_2 \geqslant \cdots$} \\
\pi_i (x_{i+1} \topLas_{s_i\alpha}) &\text{if $\alpha_i<\alpha_{i+1}$.}
\end{cases}
\end{align}
\end{lemma}
\begin{proof}
When $\alpha$ is weakly decreasing, 
our rule is immediate.
Now assume $\alpha_i < \alpha_{i+1}$ for some $i\in \Z_{>0}$.
By Corollary~\ref{C: snowy changes},
$\raj(s_i \alpha) = \raj(\alpha) - 1$.
We write $\fLb_{s_i \alpha}$
as $g + \beta^{\raj(\alpha) - 1 - |\alpha|} \: \topLas_{s_i \alpha}$
for some $g \in \Z[x_1, x_2, \cdots][\beta]$
with $\beta$-degree less than 
$\raj(\alpha) - 1 - |\alpha|$.
Now we write $\fLb_\alpha$ as
\begin{align*}
\fLb_\alpha & = \pi_i(\fLb_{s_i \alpha}) 
+ \beta \pi_i(x_{i+1}\fLb_{s_i \alpha})  \\ 
& =\pi_i(\fLb_{s_i \alpha})
+ \beta \pi_i(x_{i+1}g) + \beta^{\raj(\alpha) - |\alpha|} \pi_i(x_{i+1}\topLas_{s_i \alpha})   
\end{align*}
When we extract the coefficient of $\beta^{\raj(\alpha) - |\alpha|}$,
the left-hand side is $\topLas_\alpha$.
On the right-hand side, 
the first two terms are ignored and 
we get $\pi_i(x_{i+1}\topLas_{s_i \alpha})$.
\end{proof}

Combining Lemma~\ref{L: Top Las 1} and Lemma~\ref{L: Snowy degree},
we know $x^{\rajcode(\alpha)}$ appears in $\topLas_\alpha$
when $\alpha$ is snowy. 
Next, we show this monomial
is the leading monomial of $\topLas_\alpha$.
We start with the following observation 
about the operator $f \mapsto \pi_i(x_{i+1} f)$.
\begin{rem}
\label{R: leading}
Let $\gamma$ be a monomial.
We may describe the leading monomial of
$\pi_i(x_{i+1}x^\gamma)$ and its coefficient as follows. 
\begin{itemize}
\item If $\gamma_i > \gamma_{i+1}$,
then $x_i x^{s_i \gamma}$ is the leading monomial
with coefficient 1.
\item If $\gamma_i = \gamma_{i+1}$,
then $\pi_i(x_{i+1}x^\gamma) = 0$. 
\item If $\gamma_i < \gamma_{i+1}$,
then $x_i x^\gamma$ is the leading monomial 
with coefficient $-1$. 
\end{itemize}
\end{rem}

We can understand how the operator
$f \mapsto \pi_i(x_{i+1} f)$ changes 
the leading monomial of polynomial $f$
satisfying certain conditions. 

\begin{lemma}
\label{L: monomials from operator}
Take $f \in \Z[x_1, x_2, \cdots ]$ with $f \neq 0$.
Assume $x^\alpha$ is the leading monomial in $f$ 
with coefficient $c \neq 0$.
Pick an $i\in \Z_{>0}$ such that $\alpha_{i} > \alpha_{i+1}$.
Furthermore, assume for any monomial in $f$, 
its power of $x_{i}$ is at most $\alpha_{i}$.
Then $x_ix^{s_i \alpha}$ is the leading monomial in
$\pi_{i}(x_{i+1}f)$ with coefficient $c$. 
\end{lemma}
\begin{proof}
In this proof, we use ``$\geq$'' to denote the 
monomial order.
Let $\Gamma$ be the set of weak compositions $\gamma$
such that $x^\gamma$ appears in $f$.
Let $c_\gamma$ be the coefficient of $x^\gamma$ in $f$.
We may write $f = \sum_{\gamma \in \Gamma} c_\gamma x^\gamma$.
Then 
$\pi_{i}(x_{i+1}f) = \sum_{\gamma \in \Gamma} c_\gamma \pi_{i}(x_{i+1}x^\gamma)$.
By the remark above, $x_ix^{s_i \alpha}$ appears 
in $c_\alpha\pi_i(x_{i+1}x^\alpha)$ as the leading monomial
with coefficient $c_\alpha = c$.
It is enough to show the following claim.

\noindent\textbf{Claim:} 
Take $\gamma \in \Gamma$ such that 
$\pi_i(x_{i+1}x^{\gamma}) \neq 0$ 
(i.e. $\gamma_i \neq \gamma_{i+1}$). 
Let $x^{\gamma'}$ be the leading monomial 
in $\pi_i(x_{i+1}x^{\gamma})$.
If $x^{\gamma'}  \geq x_ix^{s_i \alpha}$,
then $\gamma = \alpha$.

\noindent\textbf{Proof:} 
Assume $\alpha \neq \gamma$.
Let $k$ be the largest index such that 
the power of $x_k$ differs in $x^{\gamma'}$ and $x_ix^{s_i \alpha}$.
By $x^{\gamma'}  \geq x_ix^{s_i \alpha}$,
the power of $x_k$ in $x^{\gamma'}$ is greater than
the power of $x_k$ in $x_ix^{s_i \alpha}$.
We must have $k \leqslant i+1$.
Otherwise,
$x^\gamma > x^\alpha$,
which contradicts $x^\alpha$ being the leading monomial in $f$.

Now we know 
$\gamma'$, $\alpha$ and $\gamma$ all agree after 
the $(i+1)^{th}$ entry.
Then
$\gamma_{i+1}'$ is at least
the power of $x_{i+1}$ in $x_ix^{s_i \alpha}$,
which is $\alpha_i$.
On the other hand, by $x^\gamma \leqslant x^\alpha$,
$\gamma_{i+1} \leqslant \alpha_{i+1}$.
Thus, 
\begin{equation}
\label{EQ: chain inequality}
\gamma_{i+1} \leq  \alpha_{i+1}<\alpha_i \leq \gamma_{i+1}'.
\end{equation}
If $\gamma_{i} < \gamma_{i+1}$,
Remark~\ref{R: leading} implies $\gamma_{i+1}'= \gamma_{i+1}$,
which is impossible.
Thus, we must have $\gamma_i > \gamma_{i+1}$. 
By Remark~\ref{R: leading} again, $\gamma_{i+1}' = \gamma_i$.
By the assumptions in the statement of the lemma, $\gamma_i \leqslant \alpha_i$,
so $\gamma_{i+1}' = \gamma_i = \alpha_i$.

Next, 
$\gamma_{i}'$ is at least
the power of $x_{i}$ in $x_ix^{s_i \alpha}$,
which is $\alpha_{i+1} + 1$.
Remark~\ref{R: leading} implies $\gamma_i' = \gamma_{i+1} + 1$.
Thus, $\gamma_{i+1} \geqslant \alpha_{i+1}$.
By~(\ref{EQ: chain inequality}), $\gamma_{i+1} = \alpha_{i+1}$.

Now we know $k < i$
and $\gamma_j = \alpha_j$ for $j = i$ or $i+1$.
Thus, $\gamma_j = \alpha_j$ for all $j > k$,
so $x^\gamma > x^\alpha$, which is a contradiction. 
\end{proof}

Now we can establish our third major lemma. 
\LasC*
\begin{proof}
We prove the result by induction on 
\begin{align*}
\ell(\alpha) := |\{(i,j) \mid\, \alpha_i < \alpha_j \text{ and } i < j\}|.
\end{align*}

If $\ell(\alpha)=0$, then $\alpha$ is weakly decreasing, 
then $\fLb_\alpha = 
x^\alpha = x^{\rajcode(\alpha)}$.
Our claim is immediate. 

Now if $\ell(\alpha) > 0$,
we can find $r$ with $\alpha_r < \alpha_{r+1}$.
Pick the largest such $r$.
For our inductive hypothesis,
assume $x^{\rajcode(s_r \alpha)}$ is the leading monomial 
of $\topLas_{s_r \alpha}$ with coefficient 1.

By the maximality of $r$, 
$\alpha_{r+1} \geqslant \alpha_{r+2} 
\geqslant \alpha_{r+3} \geqslant \cdots$.
Thus, in any $K$-Kohnert diagram of $s_r \alpha$,
there cannot be more than $\alpha_{r+1}$ cells
in row $r$.
In other words, 
for any monomial of $\topLas_{s_r \alpha}$,
the power of $x_r$ is at most $\alpha_{r+1}$.
Lemma~\ref{L: monomials from operator} implies that $x_r x^{s_r \rajcode(s_r \alpha)}$
is the leading monomial of $\topLas_{\alpha}$ with coefficient 1. 
Finally, by Corollary~\ref{C: snowy changes},
$x_r x^{s_r \rajcode(s_r \alpha)} = x^{\rajcode(\alpha)}$.
\end{proof}

\subsection{Proof of Lemma~\ref{L: Top Las similar}}
\label{SS: last major lemma}
We first derive two consequences 
of $\alpha \sim \gamma$.
We start with the following definition.
\begin{defn}
Let $D$ be a diagram.
Let $\overline{D} := \bigcup_{(r,c) \in D} [r] \times \{c\}$.
\end{defn}

In plain words, $\overline{D}$ is the diagram obtained
by filling the empty spaces above each cell of $D$.
Then $\overline{D(\alpha)}$ is completely determined 
by $\dark(\alpha)$:
\begin{lemma}
Let $\alpha$ be a weak composition.
Then $\overline{D(\alpha)} = \bigcup_{(r,c) \in \dark(\alpha)} [r] \times [c]$.
\end{lemma}
\begin{proof}
We show each side is a subset of the other. 
Take $(r_1,c_1) \in D(\alpha)$. 
By Remark~\ref{R: Basic of dark},
there is $(r_2, c_2) \in \dark(\alpha)$ such that 
$r_2 \geqslant r_1$ and $c_2 \geqslant c_1$.
Thus, $[r_1] \times \{c_1\} \subseteq [r_2] \times [c_2]$.

Take $(r_1,c_1) \in \dark(\alpha)$.
Thus, for any $c \in [c_1]$, 
$(r_1,c) \in D(\alpha)$.
Then $[r_1] \times \{c\} \subseteq \overline{D(\alpha)}$,
so $[r_1] \times [c_1] \subseteq \overline{D(\alpha)}$. 
\end{proof}

We have the following consequence of 
$\alpha \sim \gamma$.

\begin{cor}
\label{C: sim implies same overline}
If $\alpha \sim \gamma$,
then $\overline{D(\alpha)} = \overline{D(\gamma)}$.
\end{cor}
Notice that the converse is not true. 
If $\alpha = (1,2)$ and $\gamma = (0,2)$,
then $\overline{D(\alpha)} = [2] \times [2]
=  \overline{D(\gamma)}$.
However, $\alpha$ and $\gamma$ are not similar,
since $\dark(\alpha) = \{ (1,1), (2,2)\}$
and $\dark(\gamma) = \{(2,2)\}$.

Another nice consequence of $\alpha \sim \gamma$ one might expect
is $s_r \alpha \sim s_r \gamma$.
Unfortunately, this is not always true. 
It is easy to check $(0,1) \sim (1,1)$
but $s_1(0,1) = (1,0)$  and $s_1(1,1) = (1,1)$ are not similar. 
However, it is true
when $\alpha$ and $r$ satisfy 
the following condition.

\begin{lemma}
\label{L: s_i fix sim}
Let $\alpha$ be a weak composition and $r\in \Z_{>0}$.
Assume there exists $c$
such that $(r, c)\notin \snow(D(\alpha))$
but $(r+1, c)\in \snow(D(\alpha))$.
Then 
\begin{enumerate}
\item[(i)] $\alpha_{r+1} > \alpha_r$;
\item[(ii)] The diagram $\dark(s_r \alpha)$
is obtained from $\dark(\alpha)$
by switching row $r$ and row $r+1$;
\item[(iii)] For any $\gamma$ with
$\gamma \sim \alpha$,
we must have $\gamma_{r+1} > \gamma_{r}$ 
and $s_r\alpha \sim s_r\gamma$.
\end{enumerate}
\end{lemma}

\begin{proof}
Since $(r,c)$ is not in $\snow(D(\alpha))$,
we can deduce two facts:
\begin{enumerate}
\item There are no dark clouds under row $r$ in column $c$, and 
\item $\alpha_r < c$.
\end{enumerate}

By (1), the cell $(r+1,c)$ in $\snow(D(\alpha))$ 
is not a dark cloud or a snowflake.
Thus, it is unlabeled and $(r+1,c) \in D(\alpha)$.
By Remark~\ref{R: Basic of dark},
there must be a $c' > c$
such that $(r+1, c')$ is a dark cloud 
in $\snow(D(\alpha))$.
This implies $\alpha_{r+1} > c$.
By (2), 
we have $\alpha_{r+1} > \alpha_r$, 
proving (i).
Also by (2), 
the dark cloud in row $r$ of 
$\snow(D(\alpha))$,
if exists, 
is in the first $c-1$ columns. 
Thus, $\dark(s_r\alpha)$ is obtained
from $\dark(\alpha)$ by switching 
row $r$ and row $r+1$, proving (ii).

Now consider any  $\gamma \sim \alpha$.
By Proposition~\ref{P: equivalence},
$\snow(D(\gamma))$ and $\snow(D(\alpha))$
have the same underlying diagram.
By (ii), 
$\dark(s_r\gamma)$ is obtained
from $\dark(\gamma)$ by switching 
row $r$ and row $r+1$.
Since $\dark(\alpha) = \dark(\gamma)$,
we have $\dark(s_r\alpha) = \dark(s_r\gamma)$,
so $s_r \alpha \sim s_r \gamma$.
\end{proof}

These two consequences of $\alpha \sim \gamma$
allow us to prove the last main Lemma. 

\LasD*

\begin{proof}
By Lemma~\ref{L: snowy weak composition in equivalence class} 
it is enough to assume $\gamma$ is snowy, and we proceed by induction on $\raj(\alpha)$.
The base case is $\raj(\alpha) = 0$,
which implies $\alpha$ only has 0s. 
Our claim is immediate. 

Now assume $\raj(\alpha) > 0$.
Consider the diagram $\overline{D(\alpha)}$.
Clearly, the underlying diagram of any $K$-Kohnert diagram
of $\alpha$ will be a subset of $\overline{D(\alpha)}$.
In other words, any monomial in $\topLas_{\alpha}$
must divide $x^{\wt(\overline{D(\alpha)})}$.

If the underlying diagram of $\snow(D(\alpha))$
is $\overline{D(\alpha)}$, 
then $x^{\wt(\overline{D(\alpha)})}$ is the only monomial
in $\topLas_{\alpha}$.
On the other hand, 
Corollary~\ref{C: sim implies same overline} gives
$\overline{D(\alpha)} = \overline{D(\gamma)}$.
By the same argument, $x^{\wt(\overline{D(\alpha)})}$ is the only monomial
in $\topLas_{\gamma}$.
Our claim holds. 

Otherwise, we can find $(r,c) \in \overline{D(\alpha)}$
but not in $\snow(D(\alpha))$.
Choose the $(r,c)$ with the largest $r$.
First, we know $(r,c) \notin D(\alpha)$,
which implies $(r+1, c) \in \overline{D(\alpha)}$.
By the maximality of $r$,
$(r+1, c)$ is in $\snow(D(\alpha))$.
We invoke Lemma~\ref{L: s_i fix sim}
and conclude $\alpha_{r+1} > \alpha_{r}$,
$\gamma_{r+1} > \gamma_{r}$
and $s_i \alpha \sim s_i \gamma$.
Since $\gamma$ is snowy,
by Corollary~\ref{C: snowy changes},
we know $\raj(s_r \gamma) = \raj(\gamma) - 1$,
which implies $\raj(s_r \alpha) = \raj(\alpha) - 1$.
By our inductive hypothesis, 
$\topLas_{s_r \alpha} = c \topLas_{s_r \gamma}$
for some $c \neq 0$.

We may write 
$\fLb_{s_r \alpha}$ 
as $\beta^{\raj(s_r \alpha) - |\alpha|}\topLas_{s_r \alpha} + g$,
where $g$ has $\beta$-degree less than  
$\raj(s_r \alpha) - |\alpha|$.
Then 
\begin{align*}
\fLb_\alpha 
& = \pi_i(\fLb_{s_r \alpha}) + \beta \pi_i(x_{i+1}\fLb_{s_r \alpha})\\
& = \pi_i(\fLb_{s_r \alpha}) + \beta \pi_i(x_{i+1} g) + \beta^{\raj(\alpha) - |\alpha|} \pi_i(x_{i+1} \topLas_{s_r \alpha})    
\end{align*}

The first two terms on the right-hand side have $\beta$ degree 
less than $\raj(\alpha) - |\alpha|$.
Thus, the $\beta$-degree in $\fLb_\alpha$ is at most $\raj(\alpha) - |\alpha|$.
By Lemma~\ref{L: Top Las 1},
the $\beta$-degree in $\fLb_\alpha$ is $\raj(\alpha) - |\alpha|$.
Extract the coefficient of $\beta^{\raj(\alpha) - |\alpha|}$ and get
$$
\topLas_\alpha = \pi_i(x_{i+1} \topLas_{s_r \alpha}) = c\pi_i(x_{i+1} \topLas_{s_r \gamma}) = c\topLas_{\gamma},
$$
by Lemma~\ref{L: recursive top Las}.
\end{proof}

\section{Snow diagrams for Rothe diagrams}
\label{S: permutations}
Fix an $n \in \mathbb{Z}_{>0}$
throughout this section. 
We move on to study the snow diagrams 
of $RD(w)$ for $w \in S_n$. 
In subsection~\ref{SS: insertion},
we recall a version of Schensted insertion on $S_n$.
In subection~\ref{subsec.raj.rothe}, we 
show the positions of dark clouds in $\snow(RD(w))$ 
is related to the Schested insertion.
We then use this connection to 
prove that $\rajcode(RD(w))$ is consistent 
to the $\rajcode(w)$ defined in~\cite{PSW}.
In Section~\ref{subsec.shadow.diagram}, we show the dark clouds
in $\snow(RD(w))$ corresponds to
the turning points in the shadow diagram for $w$.
In Section~\ref{subsec.inverse.fireworks}, we study the
snow diagrams for inverse fireworks permutations.

\subsection{The Schensted Insertion}
\label{SS: insertion}

If a diagram is top-justified and left-justified,
we say it is a \definition{Young diagram}.
A filling of a Young diagram with positive integers
is called a \definition{tableau}.
A tableau is called \definition{partial}
if it contains distinct numbers and 
each row (resp. column) is decreasing from left to right 
(resp. top to bottom).
Notice that usually in literature, columns and rows 
are increasing. 
We reverse the convention to make our results easier to state.

The Schensted insertion~\cite{Sch}
is an algorithm defined on a partial tableau $T$ and 
a positive number $x$ that is not in $T$.
It finds the largest $x'$
in the first row of $T$ 
such that $x > x'$.
\begin{itemize}
    \item If such $x'$ does not exist,
it appends $x$ at the end of row one
and terminates.
    \item Otherwise, it replaces $x'$ by $x$
and insert $x'$ to the next row
in the same way. 
\end{itemize}
When the algorithm terminates, 
the resulting partial tableau is the output. 

For $w \in S_n$, we insert $w(n), w(n-1), \dots, w(1)$ to the empty tableau via the Schensted insertion and denote the result by $P(w)$. 

\begin{exa}
\label{E: Schensted}
Take $w \in S_7$ with one-line notation $3721564$.
The Schensted insertion on $w$ yields: 

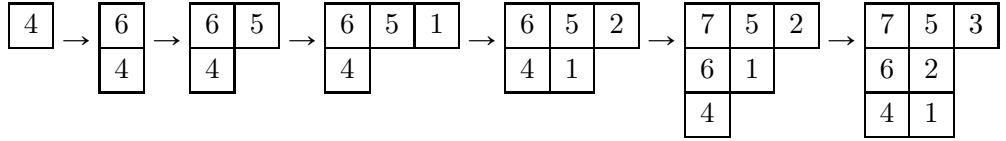
\[
\begin{ytableau}
4
\end{ytableau}
\rightarrow
\begin{ytableau}
6 \\
4
\end{ytableau}
\rightarrow
\begin{ytableau}
6 & 5\\
4
\end{ytableau}
\rightarrow
\begin{ytableau}
6 & 5 & 1\\
4
\end{ytableau}
\rightarrow
\begin{ytableau}
6 & 5 & 2\\
4 & 1
\end{ytableau}
\rightarrow
\begin{ytableau}
7 & 5 & 2\\
6 & 1\\
4
\end{ytableau}
\rightarrow
\begin{ytableau}
7 & 5 & 3\\
6 & 2\\
4 & 1
\end{ytableau}\,.
\]
\end{exa}

One classical application of the Schensted insertion
is to study increasing subsequences in a permutation.
Recall $\textup{LIS}^{w}(q)$ is the 
length of the longest increasing subsequence 
of $w\in S_n$ that starts with $q$.
It is related to the Schensted insertion as follows.

\begin{lemma} \cite[Lemma 3.3.3]{Sagan}\label{lem: c equals LIS}
Take $w\in S_n$ and
perform the Schensted insertion on $w$.
For any $r \in [n]$,
when $w(r)$ is inserted, 
it goes to column $\textup{LIS}^w(w(r))$ in row one.
\end{lemma}

\begin{exa}
Consider the $w \in S_7$ in Example~\ref{E: Schensted}.
Notice that $\textup{LIS}^w(w(4)) = 3$.
When $w(4) = 1$ is inserted to row one, 
it indeed goes to column 3. 
\end{exa}

\subsection{Rajcode of Rothe diagrams}
\label{subsec.raj.rothe}

We show
that $\rajcode(w)$ defined by Pechenik, Speyer and Weigandt
(see Definition~\ref{D: raj on permutations})
agrees with the $\rajcode(RD(w))$ (see Definition~\ref{defn.raj}). 
To do so, 
we need a better understanding of $\snow(RD(w))$. 
We start by describing how the positions of dark clouds in $\snow(RD(w))$ 
are related to the Schensted insertion described in Subsection~\ref{SS: insertion}.

\begin{prop}\label{prop: rsk and dark clouds}
Take $w \in S_n$.
Consider the Schensted insertion on $w$.
The dark cloud in row $r$ of $\snow(RD(w))$ can be described 
based on the insertion of $w(r)$.
\begin{enumerate}
\item If $w(r)$ is appended to the end of row one, 
then there is no dark cloud in the $r$\textsuperscript{th} row of $\snow(RD(w))$;
\item If $w(r)$ bumps $c$ in row one, 
then $(r, c)$ is a dark cloud in $\snow(RD(w))$.
\end{enumerate}
\end{prop}

\begin{exa}
Let $w \in S_7$ with one-line notation $3721564$.
Consider the corresponding Rothe diagram $RD(3721564)$ and its snow diagram:
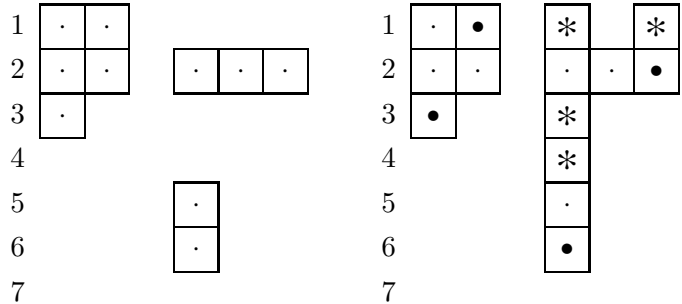
\[
\begin{ytableau}
\none[1] & \cdot & \cdot \cr
\none[2] & \cdot & \cdot & \none & \cdot & \cdot & \cdot \cr
\none[3] & \cdot \cr
\none[4] \cr
\none[5] & \none & \none & \none & \cdot \cr
\none[6] & \none & \none & \none & \cdot \cr
\none[7]
\end{ytableau}
\quad\quad
\begin{ytableau}
\none[1] & \cdot & \bullet& \none & \snowflake & \none & \snowflake \cr
\none[2] & \cdot & \cdot & \none & \cdot & \cdot & \bullet \cr
\none[3] & \bullet & \none & \none & \snowflake \cr
\none[4] & \none & \none & \none & \snowflake \cr
\none[5] & \none & \none & \none & \cdot \cr
\none[6] & \none & \none & \none & \bullet \cr
\none[7]
\end{ytableau}
\]

The Schensted insertion of $w$
is presented in Example~\ref{E: Schensted}.
We check Proposition~\ref{prop: rsk and dark clouds} in the table below.
\begin{center}
\begin{tabular}{|c |c|c|c|} 
\hline
 $r$ & $w(r)$ & insertion of $w(r)$
 in row one & position of $\bullet$
 in row $r$ of $\snow(RD(w))$\\
 \hline
 $7$ & $4$ & appended at the end of row one & row $7$ has no $\bullet$ \\ 
 \hline
 $6$ & $6$ & bumps $4$ in row one & row $6$ has $\bullet$ at $(6,4)$ \\ 
  \hline
 $5$ & $5$ & appended at the end of row one & row $5$ has no dark cloud \\ 
  \hline
 $4$ & $1$ & appended at the end of row one & row $4$ has no dark cloud \\ 
  \hline
 $3$ & $2$ & bumps $1$ in row one & row $3$ has $\bullet$ at $(3,1)$  \\ 
  \hline
 $2$ & $7$ & bumps $6$ in row one & row $2$ has $\bullet$ at $(2,6)$ \\ 
  \hline
 $1$ & $3$ & bumps $2$ in row one & row $1$ has $\bullet$ at $(1,2)$  \\ 
 \hline
\end{tabular}
\end{center}
\end{exa}

\begin{proof}
We prove the statement by induction on $r$ starting from $r=n$.
The number $w(n)$ is inserted into the empty tableau. 
In this case, it is appended to the end of the first row. 
It is also clear that there can not be any dark cloud on row $n$ of $\snow(RD(w))$.

Now suppose the statement holds for $r+1,r+2,\dots,n$ for some $r \leqslant n-1$. 
Let $P$ be the tableau right before the insertion of $w(r)$.
By the inductive hypothesis, for each $r' > r$, $w(r')$ appears in row 1 of $P$
if and only if there is no dark cloud in column $w(r')$ under row $r$ of $\snow(RD(w))$.
Now consider the insertion of $w(r)$.
\begin{enumerate}
\item Case 1: $w(r)$ is appended to the end of row 1. 

Assume toward contradiction that $(r, w(r'))$ is a dark cloud of $\snow(RD(w))$ for some $r' > r$. 
Then $w(r) > w(r')$. 
Moreover, there is no dark cloud in the column of $w(r')$ under row $r$,
so $w(r')$ is in row $1$ of $P$. 
Thus, $w(r)$ cannot be appended in row $1$, a contradiction.
\item Case 2: $w(r)$ bumps $w(r')$ in row 1 for some $r' > r$.

Then $w(r) > w(r')$.
The cell $(r, w(r'))$ is in $RD(w)$. We need to show that it is a dark cloud in $\snow(RD(w))$.
By Remark~\ref{R: Basic of dark}, we just need to make sure
there is no dark cloud under it or on its right.

Suppose that there is a dark cloud in column $w(r')$ under row $r$. By the inductive hypothesis, $w(r')$ cannot appear in row 1 of $P$, which is a contradiction.

Finally, suppose there is a dark cloud on the right of $(r, w(r'))$.
We may write this dark cloud as $(r, w(r''))$ with $w(r'') > w(r')$.
Since it is a cell in $RD(w)$,
we also have $r'' > r$ and $w(r) > w(r'')$.
Since it is a dark cloud, there is no dark cloud under it.
By the inductive hypothesis, $w(r'')$ is in row 1 of $P$.
This is a contradiction: $w(r)$ should bump $w(r'')$ instead of $w(r')$ 
since $w(r) > w(r'') > w(r')$. \qedhere
\end{enumerate}
\end{proof}

\begin{ourthm}
\label{T: equivalence of rajcode on perm}
For $w \in S_n$, 
$\rajcode(w) = \rajcode(RD(w))$.
\end{ourthm}

\begin{proof}
Take $r \in [n]$.
Consider row $r$ of $\snow(RD(w))$.
It contains $\invcode(w)_r$ cells that are not snowflakes. 
Let $d_r$ be the number of dark clouds in $\snow(RD(w))$ 
that are southeast of $(r,w(r))$.
Clearly, $d_r$ is also the number of snowflakes in row $r$ of $\snow(RD(w))$.
We have $\rajcode(RD(w))_r = \invcode(w)_r + d_r$.

Consider the Schensted insertion of $w$.
Let $P$ be the tableau right before the insertion of $w(r)$.
Define $A$ as the number of elements
in $P$ that are larger than $w(r)$.
We compute $A$ in two ways. 
\begin{itemize}
\item The tableau $P$
consists of numbers 
$w(r+1), \dots, w(n)$.
There are $\invcode(w)_r$ of them 
less than $w(r)$,
so $A = n - r - \invcode(w)_r$.
\item
Assume when inserting $w(r)$ to $P$,
it goes to column $c$ of row 1. 
Thus,  
$c-1$ is the number of entries in row 1
of $P$ that are larger than $w(r)$.
By Proposition~\ref{prop: rsk and dark clouds},
$d_r$ is the number of entries 
under row 1 of $P$
that are larger than $w(r)$.
We have $A = c - 1 + d_r$.
By Lemma~\ref{lem: c equals LIS},
$c = \textup{LIS}^w(w(r))$,
so $A = \textup{LIS}^w(w(r)) - 1 + d_r$.
\end{itemize}

Combining the two expressions of $A$ 
yields
$$
 n - r - \invcode(w)_r = \textup{LIS}^w(w(r)) - 1 + d_r, \textrm{ so}
$$
$$
\rajcode(RD(w))_r = \invcode(w)_r + d_r
= n - r + 1 - \textup{LIS}^w(w(r))
= \rajcode(w)_r.
$$
\end{proof}

\subsection{Dark Clouds of the Rothe Diagram via Viennot's geometric construction}
\label{subsec.shadow.diagram}
In 1977, Xavier G{\'e}rard Viennot gave a diagrammatic construction of the RSK correspondence 
in terms of shadow lines (\cite{viennot}). It is also known as the matrix-ball construction. We will show that the dark clouds in the snow diagram of a permutation can be obtained via Viennot's geometric construction. 
We denote $\Rowone{w}$ to be the first row of the tableau obtained by Schensted insertion on $w$.

For two cells $(i,j), (m,n) \in \NN\times \NN$, 
$(m,n)$ lies in the shadow of $(i,j)$ if and only if $m \leq i$ and $n \leq j$. This can be visualized by imagining shedding light from the Southeast. 
\footnote{The usual convention can be thought of shedding light from the Northwest, which corresponds to the usual Schensted insertion. We reverse the direction to match our decreasing Schensted insertion convention.}
To obtain the \definition{shadow diagram} of $w\in S_n$, consider the points $(1,w(1)),\dots, (n,w(n))$. Let $\left(i^{(1)}_1,w(i^{(1)}_1)\right),\dots,\left(i^{(1)}_{\ell_{1}},w(i^{(1)}_{\ell_1})\right)$ be the points that are not in the shadow of any other point for some $\ell_1 \geqslant 1$ and $i^{(1)}_1> i^{(1)}_2>\dots > i^{(1)}_{\ell_1}$. Then the first \definition{shadow line} $L_1(w)$ is the boundary of the combined shadows of the points $\left(i^{(1)}_1,w(i^{(1)}_1)\right),\dots,\left(i^{(1)}_{\ell_1},w(i^{(1)}_{\ell_1})\right)$. The rest of the $L_j(w)$ can be constructed recursively. Supposed $L_1,\dots,L_{j-1}$ have been constructed, remove all points in the set $$\left\{ \left(i^{(p)}_{k},w(i^{(p)}_k)\right): 1\leqslant p \leqslant j-1, 1\leqslant k \leqslant \ell_p \right\},$$ then $L_j$ is the boundary of the shadow of the remaining points of the points left, which we label as 
\[
\left(i^{(j)}_1,w(i^{(j)}_1)\right),\dots \left(i^{(j)}_{\ell_j},w(i^{(j)}_{\ell_j})\right),
\]
for some $\ell_j \geqslant 1$ and $i^{(j)}_1>i^{(j)}_2>\dots>i^{(j)}_{\ell_j}$. Once there is no point left, the shadow lines we obtained form the shadow diagram for $w$.

\begin{ourthm}[\cite{viennot}]
Given $w\in S_n$ and suppose $L_1,\dots,L_s$ are the shadow lines obtained from $w$ until there is no point left. Then $s$ equals the size of $\Rowone{w}$.
\end{ourthm}

For each shadow line $L_j$, it also consists $\ell_{j}-1$ ``turning points'', which are points $(x,y)$ of $L_j$ such that $(x-1,y), (x,y-1)\notin L_j$, i.e.,
\[
\left(i^{(j)}_2,w(i^{(j)}_1)\right), \left(i^{(j)}_3,w(i^{(j)}_2)\right), \dots, \left(i^{(j)}_{\ell_j},w(i^{(j)}_{\ell_j-1})\right).
\]
In total, there are $n-|\Rowone{w}|$ turning points for each $w\in S_n$.
There is a classical result connecting
these turning points to the Schensted insertion. 

\begin{ourthm}[\cite{viennot,knuth}]\label{thm: shadow turn}
Let a shadow line $L_j$ of a permutation $w$ consists of points
\[
\left(i^{(j)}_1,w(i^{(j)}_1)\right),\dots \left(i^{(j)}_{\ell_j},w(i^{(j)}_{\ell_j})\right)
\]
for some $\ell_j \geqslant 1$ and $i^{(j)}_1>i^{(j)}_2>\dots>i^{(j)}_{\ell_j}$. 
Then during Schensted insertion on $w$, when we insert $w(i^{(j)}_{k+1})$,
it bumps $w(i^{(j)}_{k})$ from the first row.
\end{ourthm}
 
Combining Proposition~\ref{prop: rsk and dark clouds} and Theorem~\ref{thm: shadow turn}, we have the following.

\begin{cor}
Each of the turning points in the shadow diagram of $w$ contains a dark cloud in $\snow(w)$. Any dark cloud in $\snow(w)$ is also a turning point in the shadow diagram of $w$.
\end{cor}

\begin{exa}
Consider $w = 3721564\in S_7$. We present its Rothe diagram, its shadow diagram, 
and the snow diagram of $RD(w)$. 
From Example~\ref{E: Schensted},
the Schensted insertion on $w$ yields a tableau
whose row 1 has three cells. 
Correspondingly, there are three shadow lines.
The turning points of the shadow lines are $(3,1), (1,2), (6,4), (2,6)$, which are positions for dark clouds in $\snow(RD(w))$.

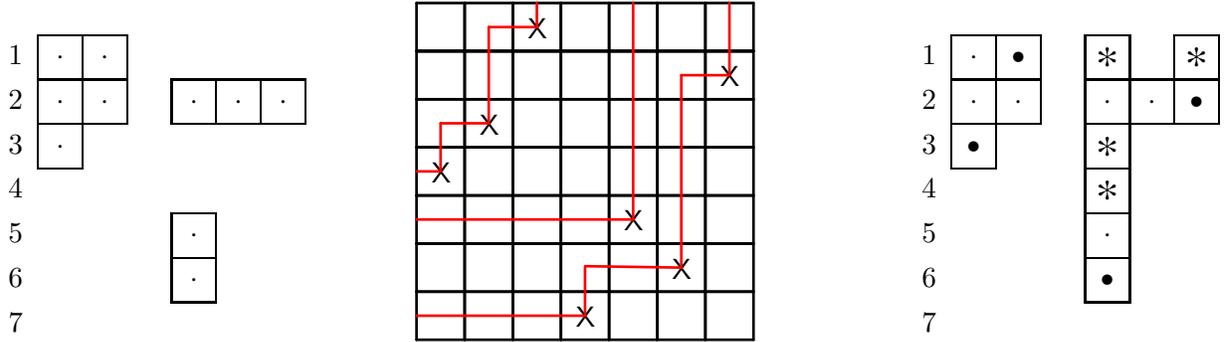
\begin{figure}[ht]
\centering
\begin{minipage}{0.26\textwidth}
\[
\begin{ytableau}
\none[1] & \cdot & \cdot \cr
\none[2] & \cdot & \cdot & \none & \cdot & \cdot & \cdot \cr
\none[3] & \cdot \cr
\none[4] \cr
\none[5] & \none & \none & \none & \cdot \cr
\none[6] & \none & \none & \none & \cdot \cr
\none[7]
\end{ytableau}
\]
\end{minipage}
\begin{minipage}{0.46\textwidth}
\begin{tikzpicture}[line cap=round,line join=round,>=triangle 45,x=0.64cm,y=0.64cm]
\clip(-2,-2) rectangle (13.5,8.5);
\draw [line width=1pt,color=black] (0,0)-- (7,0);
\draw [line width=1pt,color=black] (7,0)-- (7,7);
\draw [line width=1pt,color=black] (7,7)-- (0,7);
\draw [line width=1pt,color=black] (0,7)-- (0,0);
\draw [line width=1pt,color=black] (0,6)-- (7,6);
\draw [line width=1pt,color=black] (0,6)-- (7,6);
\draw [line width=1pt,color=black] (0,5)-- (7,5);
\draw [line width=1pt,color=black] (0,4)-- (7,4);
\draw [line width=1pt,color=black] (0,3)-- (7,3);
\draw [line width=1pt,color=black] (0,2)-- (7,2);
\draw [line width=1pt,color=black] (0,1)-- (7,1);
\draw [line width=1pt,color=black] (1,0)-- (1,7);
\draw [line width=1pt,color=black] (2,7)-- (2,0);
\draw [line width=1pt,color=black] (3,7)-- (3,0);
\draw [line width=1pt,color=black] (4,0)-- (4,7);
\draw [line width=1pt,color=black] (5,7)-- (5,0);
\draw [line width=1pt,color=black] (6,0)-- (6,7);
\draw (2.1,6.9) node[anchor=north west] {$\X$};
\draw (6.1,5.9) node[anchor=north west] {$\X$};
\draw (1.1,4.9) node[anchor=north west] {$\X$};
\draw (0.1,3.9) node[anchor=north west] {$\X$};
\draw (4.1,2.9) node[anchor=north west] {$\X$};
\draw (5.1,1.9) node[anchor=north west] {$\X$};
\draw (3.1,0.9) node[anchor=north west] {$\X$};
\draw [line width=1pt,color=black] (0,6)-- (7,6);
\draw [line width=1pt,color=black] (0,6)-- (7,6);
\draw [line width=1pt,color=black] (0,5)-- (7,5);
\draw [line width=1pt,color=black] (0,4)-- (7,4);
\draw [line width=1pt,color=black] (0,3)-- (7,3);
\draw [line width=1pt,color=black] (0,2)-- (7,2);
\draw [line width=1pt,color=black] (0,1)-- (7,1);
\draw [line width=1pt,color=black] (1,0)-- (1,7);
\draw [line width=1pt,color=black] (2,7)-- (2,0);
\draw [line width=1pt,color=black] (3,7)-- (3,0);
\draw [line width=1pt,color=black] (4,0)-- (4,7);
\draw [line width=1pt,color=black] (5,7)-- (5,0);
\draw [line width=1pt,color=black] (6,0)-- (6,7);
\draw [line width=1pt,color=red] (0,0.5)-- (3.5,0.5);
\draw [line width=1pt,color=red] (3.5,0.5)-- (3.5,1.5);
\draw [line width=1pt,color=red] (3.5,1.53)-- (5.5,1.5);
\draw [line width=1pt,color=red] (5.5,1.5)-- (5.5,5.5);
\draw [line width=1pt,color=red] (5.5,5.5)-- (6.5,5.5);
\draw [line width=1pt,color=red] (6.5,5.5)-- (6.5,7);
\draw [line width=1pt,color=red] (0,2.5)-- (4.5,2.5);
\draw [line width=1pt,color=red] (4.5,2.5)-- (4.5,7);
\draw [line width=1pt,color=red] (0,3.5)-- (0.5,3.5);
\draw [line width=1pt,color=red] (0.5,3.5)-- (0.5,4.5);
\draw [line width=1pt,color=red] (0.5,4.5)-- (1.5,4.5);
\draw [line width=1pt,color=red] (1.5,4.5)-- (1.5,6.5);
\draw [line width=1pt,color=red] (1.5,6.5)-- (2.5,6.5);
\draw [line width=1pt,color=red] (2.5,6.5)-- (2.5,7);
\end{tikzpicture}
\end{minipage}
\begin{minipage}{0.26\textwidth}
\[
\begin{ytableau}
\none[1] & \cdot & \bullet& \none & \snowflake & \none & \snowflake \cr
\none[2] & \cdot & \cdot & \none & \cdot & \cdot & \bullet \cr
\none[3] & \bullet & \none & \none & \snowflake \cr
\none[4] & \none & \none & \none & \snowflake \cr
\none[5] & \none & \none & \none & \cdot \cr
\none[6] & \none & \none & \none & \bullet \cr
\none[7]
\end{ytableau}
\]
\end{minipage}
\caption{Left: RD($w$);\quad Middle: shadow diagram of $w$;\quad Right: $\snow(RD(w))$}
\end{figure}

\end{exa}

\begin{rem}
A geometric interpretation for the rajcode is given 
in~\cite[Section~4]{PSW} in terms of the ``blob diagrams.'' 
Specifically, the set of points in the same shadow line in the 
shadow line diagram is labeled as $B_n,B_{n-1},\dots$ from
southeast to northwest. With the labeling on the blob diagrams, we can
obtain the rajcode directly. 
That is, if $(i,w(i))\in B_k$, then $\rajcode(w)_i=k-i$.
\end{rem}

\subsection{Inverse fireworks permutations}
\label{subsec.inverse.fireworks}

Now we have seen that our snow diagrams 
are connected to the work of Pechenik, Speyer and Weigandt~\cite{PSW}.
We recall another interesting notion in their work.

\begin{defn}\cite[Definition 3.5]{PSW}
A permutation $w \in S_n$ is a 
\definition{fireworks permutation}
if its initial element in each decreasing run
is increasing. 
A permutation $w \in S_n$ is an 
\definition{inverse fireworks permutation}
if $w^{-1}$ is a fireworks permutation.
\end{defn}

Inverse fireworks permutations
are the representatives of equivalence classes,
given by permutations with the same $\rajcode$~\cite{PSW}.
The snowy weak compositions play the same role
in our study of $\topLas$.
We investigate the similarities between inverse fireworks permutations
and snowy weak compositions. 
For $w$ inverse fireworks, $RD(w)$ enjoy analogous properties as the $D(\alpha)$ of snowy $\alpha$. 
We start with the following 
observation about $RD(w)$.

\begin{lemma}
\label{L: inv fire perm diagram}
Let $w \in S_n$ be an inverse fireworks permutation.
Consider each $r \in [n]$ such that 
row $r$ of $RD(w)$ is not empty. 
The rightmost cell in row $r$ of $RD(w)$
is $(r, w(r) - 1)$.
\end{lemma}
\begin{proof}
Recall that $(r, w(r')) \in RD(w)$ 
if and only if $(r, r') \in \Inv(w)$
if and only if $(w(r'), w(r)) 
\in \Inv(w^{-1})$.
Let $c = w(r)$.
Clearly, cells in row $r$
of $RD(w)$ are within the first $c - 1$
columns.
It remains to check $(r, c-1) \in RD(w)$,
which is equivalent to $(c -1, c) \in \Inv(w^{-1})$.

Since row $r$ of $RD(w)$ is nonempty, 
it must contain a cell $(r,i)$ such that $(i,c)\in \Inv(w^{-1})$ 
for some $i \in [c-1]$. Since $w^{-1}(i)>w^{-1}(c)$ and $w^{-1}$ is 
fireworks, $w^{-1}(c)$ can not be the initial element in its 
decreasing run. Therefore $w^{-1}(c-1) > w^{-1}(c)$ and 
we have $(c-1, c) \in \Inv(w^{-1})$.
\end{proof}

We can characterize the inverse fireworks permutations
using Rothe diagrams or the snow diagram of the permutation. This 
is similar to Remark~\ref{R: snowy},
where we describe snowy weak compositions using key diagrams and 
dark clouds.

\begin{prop}
Take $w \in S_n$.
The following are equivalent:
\begin{enumerate}
\item $w$ is an inverse fireworks permutation.
\item In $RD(w)$, the rightmost cells in each row 
are in different columns. 
\item In $\snow(RD(w))$, 
the rightmost cell in each row is a dark cloud. 
\end{enumerate}
\end{prop}
\begin{proof}
The last two statements are clearly equivalent. 
Now we establish the equivalence of 
the first two statements.

Assume $w$ is inverse fireworks.
Take $r, r' \in [n]$ with $r \neq r'$
such that row $r$ and row $r'$ of $RD(w)$
are not empty.
By Lemma~\ref{L: inv fire perm diagram},
the rightmost cell in row $r$ 
(resp. $r'$)
is at $(r, w(r) - 1)$ 
(resp. $(r', w(r') - 1)$).
Clearly, $w(r) - 1 \neq w(r') - 1$,
so we have our second statement. 

Now we assume $w$ is not inverse fireworks.
We can find a number $r$ in $w^{-1}$
such that $r$ is the initial element in
its decreasing run, 
but $r$ is less than $r'$, 
the initial element of the 
previous decreasing run. 
Let $c' = w(r')$ and $c = w(r)$.
Since $(c', c) \in \Inv(w^{-1})$, 
$(r, c') \in RD(w)$.
Thus, row $r$ of $RD(w)$ is not empty.
Let $(r, i)$ be the rightmost cell
in row $r$.
In other words, $i$ is the largest
such that $(i, c) \in \Inv(w^{-1})$.
We have $c' \leqslant i < c - 1$.
Consider the decreasing run before $w^{-1}(c)$:
$w^{-1}(c') > w^{-1}(c' + 1) > \cdots > w^{-1}(c - 1)$.
We see $(i, i+1)$ is also in $\Inv(w^{-1})$.
In row $w^{-1}(i+1)$, 
the cell $(w^{-1}(i+1), i)$ is the 
rightmost cell of its row.
Thus, the second statement does not hold,
and the proof is finished. 
\end{proof}

With the above proposition, we can compute $\rajcode(w)$ easily
if $w$ is inverse fireworks.
The following rule is similar to Lemma~\ref{L: dark, rajcode and raj of snowy weak composition}(\ref{snowy.weak.2}).

\begin{prop}
Assume $w \in S_n$ is inverse fireworks.
For each $r \in [n]$,
\begin{align*}
\rajcode(w)_r = |\{r' > r: &(r,r') \in \Inv(w) \textrm{ or }\\
&w(r') > w(r)\textrm{ and }(r', r'') \in \Inv(w) \textrm{ for some } r''\}|.
\end{align*}
\end{prop}

\begin{proof}
First, we know $\rajcode(w)_r = \rajcode(RD(w))_r$
is the number of cells in the $r$\textsuperscript{th}  row of $\snow(RD(w))$.
The number of non-snowflake cells
on this row is given by $|\{r': (r, r') \in \Inv(w)\}|$.

Now we count the number of snowflakes 
in row $r$ of $\snow(RD(w))$.
It is the number of $r' > r$
such that row $r'$ of $\snow(RD(w))$ has a dark cloud
on the right of the column $w(r)$.
By Lemma~\ref{L: inv fire perm diagram}, 
row $r'$ has a dark cloud 
at column $w(r') - 1$ if $RD(w)$ is nonempty in row $r'$.
Thus, the number of snowflakes 
in row $r$ of $\snow(RD(w))$ is the number of $r' > r$
such that $w(r') > w(r)$ and $(r', r'') \in \Inv(w)$
for some $r''$.
\end{proof}

\section{Vector space spanned by \texorpdfstring{$\topGro_w$}{TEXT}}
\label{S: Space}

We now study the vector spaces
$\widehat{V}_n := \Q\textrm{-span}\{\topGro_w: w \in S_n\}$
and $\widehat{V} := \Q\textrm{-span}\{\topGro_w: w \in S_+\}$.
By Theorem~\ref{thm:psw},
they have bases
$$\{\topGro_w: w \in S_n 
\textrm{ is inverse fireworks}\}  
\quad\textrm{ and }\quad 
\{\topGro_w: w \in S_+ \textrm{ is inverse fireworks}\}
$$
respectively. 
By~\cite{Cla}, 
the number of inverse fireworks permutations in $S_n$
is $B_n$, the $n$\textsuperscript{th} Bell number.
Thus, $\widehat{V}_n$ has dimension $B_n$.

We introduce another basis of $\widehat{V}_n$ and $\widehat{V}$ 
consisting of $\topLas_\alpha$, the top-degree components of Lascoux polynomials.
One application of 
the top Lascoux basis 
is to compute the Hilbert series of $\widehat{V}_n$
and $\widehat{V}$.
For a vector space $V \subseteq \Q[x_1, x_2, \cdots]$,
the \definition{Hilbert series} of $V$ is 
$$\Hilb(V;q) := \sum_{d \geq 0}\, m_d q^d,$$
where $m_d$ is the number of polynomials 
with degree $d$ in a homogeneous 
basis of $V$.

In Subsection~\ref{SS: Bell},
we recall the definition of $B_n$ and its $q$-analogue $B_n(q)$. 
In Subsection~\ref{subsec.hilbert.vn}, we compute $\Hilb(\widehat{V}_n;q)$ 
using $B_n(q)$ and rook-theoretic results. 
In Subsection~\ref{subsec.v.hat},
we compute $\Hilb(\widehat{V};q)$.

\subsection{Stirling numbers, Bell numbers
and their \texorpdfstring{$q$}{TEXT}-analogues}
\label{SS: Bell}

Let $n, k$ be non-negative integers
throughout this subsection. 
Let $S_{n, k}$ be the \definition{Stirling number
of the second kind},
defined by the recurrence relation
$$
S_{n+1,k} = S_{n, k - 1} + k S_{n, k},
$$
together with $S_{0, 0} = 1$
and $S_{0,k} = 0$ if $k > 0$.
Let $B_n:= \sum_{j = 0}^n S_{n,j}$
be the Bell number
which satisfies the following recurrence relation
$$
B_{n + 1} = \sum_{j = 0}^n {\binom{n}{j}} B_j.
$$
 
Let $\Rook_n$ be the set 
of non-attacking rook diagrams contained 
in $\Stair_n$.
It is an exercise to show
$B_n = |\Rook_n|$.
In~\cite{BCHR},
Butler, Can, Haglund, and Remmel
built an explicit bijection between
$\Rook_n$ and set partitions of $[n]$ .

Now consider the polynomial ring $\Q[q]$.
Define $[n]_q := 1 + q + \cdots + q^{n-1}$.
Define a \definition{$q$-analogue of $S_{n,k}$} recursively by:
$$
S_{n+1, k}(q) = q^{k-1} S_{n, k-1}(q) + [k]_q S_{n, k}(q),
$$
with base cases $S_{0, k}(q) = S_{0,k}$.
Similarly, define a \definition{$q$-analogue of $B_n$} by $B_n(q) := \sum_{j = 0}^n S_{n,j}(q)$.
The coefficients in $B_{n}(q)$
are given in OEIS A126347.
By~\cite{Wa},
$B_n(q)$ satisfies the recurrence relation
$$
B_{n+1}(q) = \sum_{j = 0}^n q^j {\binom{n}{j}} B_j(q).
$$

Milne~\cite{Mil} first gave a combinatorial
model for $S_{n,k}(q)$ using set partitions. 
We use the combinatorial model
developed by Garsia and Remmel~\cite{GR}.
They defined a statistic on $\Rook_n$
called ``inversion''.
We rename it as $\GR_n$ to distinguish it 
from the inversion on permutations. 
\begin{defn}[\cite{GR}]
Assume $R \in \Rook_n$.
For each $(r,c) \in R$,
mark all cells $(r', c)$ with $r' \in [r]$ in $\Stair_n$.
Also, mark all cells $(r, c')$ with $c' \in [c]$ in $\Stair_n$.
The number $\GR_n(R)$ counts cells in $\Stair_n$
that are not marked. 
\end{defn}

Garsia and Remmel prove that
\begin{equation}
\label{EQ: GR Stirling}
S_{n,k}(q) = \sum_{\substack{D \in \Rook_{n}\\ |D| = n-k}} q^{\GR_n(D)},
\end{equation}
which implies
\begin{equation}
\label{EQ: GR Bell}
B_n(q) = \sum_{D \in \Rook_n} q^{\GR_n(D)}.
\end{equation}

From this formula, $B_n(q)$ has degree ${\binom{n}{2}}$
since $\GR_n(\emptyset) = {\binom{n}{2}}$.

\subsection{Computing \texorpdfstring{$\Hilb(\widehat{V}_n; q)$}{TEXT}}
\label{subsec.hilbert.vn}

Define 
$$C_n := \{ \alpha \in C_+: \supp(\alpha) \subseteq [n - 1], \alpha_i \leq n - i \textrm{ for all $i \in [n - 1]$}\}.$$

Then we can refine 
Theorem~\ref{T: Gro to Las}.
\begin{cor}
\label{C: Gro to Las for n}
For $w \in S_n$, 
$\fGb_w$ expands positively into $\{\fLb_\alpha: \alpha \in C_n\}$.
\end{cor}
\begin{proof}
By Theorem~\ref{T: Gro to Las},
we can expand $\fGb_w$ into a sum of Lascoux polynomials.
We just need to make sure for each $\fLb_{\alpha}$ appearing
in the expansion with a nonzero coefficient,
we have $\alpha \in C_n$.

We know the monomial $x^\alpha$ is the leading monomial of 
$\kappa_\alpha$, so $x^\alpha$ appears in $\fLb_\alpha$.
Since all coefficients in the sum are positive,
we know $x^\alpha \beta^m$ appears in $\fGb_w$ for
some $m \in \mathbb{Z}_{\geq 0}$.
By the monomial expansion of $\fGb_w$ 
given by Fomin and Kirillov~\cite{FK},
we have $\alpha \in C_n$.
\end{proof}

By this corollary and Lemma~\ref{L: Expands positively},
we have the following. 
\begin{cor}
\label{C: Top Gro to Top Las for n}
For $w \in S_n$, 
$\topGro_w$ expands positively 
into $\{\topLas_\alpha: \alpha \in C_n\}$.
\end{cor}

Now we are ready to give another basis of $\widehat{V}_n$:
\begin{prop}
\label{P: Top Las basis}
The space $\widehat{V}_n$ is also $\Q\textrm{-span}\{\topLas_\alpha: \alpha \in C_n\}$. 
It has a basis 
$\{ \topLas_\alpha: \alpha \in C_n \textrm{ is snowy}\}$.
\end{prop}

\begin{proof}
By Corollary~\ref{C: Top Gro to Top Las for n} ,
$\widehat{V}_n$ is a subspace of 
$\Q\textrm{-span}\{\topLas_\alpha: \alpha \in C_n\}$.
By Lemma~\ref{L: snowy weak composition in equivalence class},
for any $\alpha \in C_n$,
we can find a snowy $\gamma \in C_n$
such that $\gamma \sim \alpha$.
Then by Theorem~\ref{T: Top Las},
$\topLas_\alpha$ is a scalar
multiple of $\topLas_\gamma$.
Thus, 
$\Q\textrm{-span}\{\topLas_\alpha: \alpha \in C_n\}$ is a subspace of 
$\Q\textrm{-span}\{\topLas_\alpha: \alpha \in C_n \textrm{ is snowy}\}$.
Notice that $\{\topLas_\alpha: \alpha \in C_n \textrm{ is snowy}\}$
is linear independent since its polynomials
have distinct leading terms 
by Theorem~\ref{T: Top Las}.

By~\cite[Theorem~1.4]{PSW}
$\widehat{V}_n$ has dimension $B_n$.
It remains to check the number of snowy weak compositions
in $C_n$ is also $B_n$.
In Lemma~\ref{L: snowy biject rook},
we show $\dark(\cdot)$ a bijection from
snowy weak compositions in $C_+$ to $\Rook_+$.
Clearly it restricts to a bijection from
$C_n$ to $\Rook_n$,
which has size $B_n$.
\end{proof}

We use the top Lascoux basis 
to derive $\Hilb(\widehat{V}_n; q)$. 
Let us translate 
the statistic $\raj(\cdot)$ 
on snowy weak compositions
to non-attacking rook diagrams.

\begin{defn}
Take $R \in \Rook_+$.
Define the \definition{Northwest number} of $R$,
denoted as $\NW(R) := \raj(\alpha)$, 
where $\alpha$ is any weak composition
with $\dark(\alpha) = R$.
\end{defn}

Equivalently, we may compute $\NW(R)$ as follows:
For each $(r,c) \in R$, we mark all cells
weakly above it and to its left.
By Lemma~\ref{L: dark recovers snow},
these marked cells agree with the underlying diagram of $\snow(D(\alpha))$ for any $\alpha$ with $\dark(\alpha) = R$.
Then $\NW(R)$ is just the number of marked cells.
Comparing this statistic with $\GR_n(\cdot)$
defined in Subsection~\ref{SS: Bell},
we have the following connection.

\begin{rem}
\label{R: NW and GR}
Take $R \in \Rook_n$.
Then $\GR_n(R) = |\Stair_n| - \NW(R) = {\binom{n}{2}} - \NW(R)$.
\end{rem}

Finally, we can derive an expression
for the degree generating function of $\widehat{V}_n$.

\begin{prop}
\label{T: F_n}
We have
$$
\Hilb(\widehat{V}_n;q) 
= q^{{\binom{n}{2}}} B_{n}(q^{-1})
= \rev(B_n(q)),
$$
where $\rev(\cdot)$ is the operator that
reverse the coefficients of a polynomial.
In other words, it sends a polynomial $f(q)$
of degree $d$ to $q^d f(q^{-1})$.
\end{prop}
\begin{proof}
By Prop~\ref{P: Top Las basis},
$
\Hilb(\widehat{V}_n;q) = \quad \sum_{\alpha} q^{\raj(\alpha)}$
where the sum is over snowy $\alpha \in C_n$.
Apply the bijection $\dark(\cdot)$
to $\alpha$ in the summation, we have
\begin{align*}
\Hilb(\widehat{V}_n;q) & = \sum_{R \in \Rook_n} q^{\NW(R)} =\sum_{R \in \Rook_n} q^{{\binom{n}{2}} - \GR_n(R)}\\
&=q^{\binom{n}{2}}\sum_{R \in \Rook_n} q^{-\GR_n(R)} = q^{{\binom{n}{2}}} B_{n}(q^{-1})\,,
\end{align*}
where the second equality is by 
Remark~\ref{R: NW and GR}
and the last equality is by (\ref{EQ: GR Bell}).
Since $B_n(q)$ has degree ${\binom{n}{2}}$, we have $\Hilb(\widehat{V}_n;q) = \rev(B_n(q))$.
\end{proof}

\subsection{Computing  \texorpdfstring{$\Hilb(\widehat{V};q)$}{TEXT}}
\label{subsec.v.hat}

First, we show the top Lascoux polynomials also span the space $\widehat{V}$.
\begin{prop}
\label{P: V span}
We have $\Q\textrm{-span}\{\topLas_\alpha \mid \alpha \in C_+\} = \widehat{V}$.
\end{prop}
\begin{proof}
By Corollary~\ref{C: top Gro to top Las},
$\widehat{V}$ is in the $\Q\textrm{-span}$ of 
$\{\topLas_\alpha: \alpha \in C_+\}$.
Now consider $\alpha \in C_+$.
There exists $n$ large enough such that
$\alpha \in C_n$.
Then $\topLas_\alpha \in \widehat{V}_n \subset \widehat{V}$.
\end{proof}

\begin{cor}
The space $\widehat{V}$
has a basis 
$\{ \topLas_\alpha: \alpha \in C_+ \textrm{ is snowy}\}.$
\end{cor}

With the top Lascoux basis, we have
$$
\Hilb(\widehat{V};q)  \quad = \quad \sum_{\substack{\alpha \: \in \: C_+, \\ 
\alpha \textrm{ is snowy}}} q^{\raj(\alpha)} \quad
= \quad \sum_{R \in \Rook_+} q^{\NW(R)},$$
where the second equality is obtained by 
applying $\dark(\cdot)$ on $\alpha$ in the second expression.
On the other hand, since $\widehat{V} = \bigcup_{n \geq 1} \widehat{V}_n$, 

 $\Hilb(\widehat{V};q)$ is the limit of $\Hilb(\widehat{V}_n;q)$
as $n$ goes to infinity. 
According to OEIS, 
coefficients in $B_n(q)$ are in \href{https://oeis.org/A126347}{A126347} and 
the coefficients of $\Hilb(\widehat{V};q)$ are in \href{https://oeis.org/A126348}{A126348}. 
A formula for $\Hilb(\widehat{V};q)$ in OEIS is given by Jovovic: 
$\prod_{m > 0} (1 + \frac{q^m}{1 - q})$.
For completeness, we check this rule using our formula
of $\Hilb(\widehat{V};q)$ involving snowy weak compositions.

\begin{prop}
\label{T: F}
We have
$$\Hilb(\widehat{V};q) = \sum_{\alpha \text{ is snowy}} q^{\raj(\alpha)}
= \prod_{m > 0} \left( 1 + \frac{q^m}{1 - q} \right)
$$
\end{prop}
\begin{proof}
Let $\mathsf{snowy}(M)$ be the set of all snowy weak compositions
with the largest entry being at most $M$.
It suffices to show 
$$ \sum_{\alpha \in \mathsf{snowy}(M)} q^{\raj(\alpha)}=\prod_{m  > 0}^M\left(1 + \frac{q^m}{1 - q}\right).$$ 
We prove it by induction on $M$. The claim is immediate when $M = 0$ as both sides are $1$.

Now assume the claim above holds for some $M\geq 0$.
Let $\mathsf{snowy}(M)_i$ be the set of all snowy weak compositions
$\alpha$ such that its largest entry is $\alpha_i = M$.
With this notation, we can express $\mathsf{snowy}(M)$ recursively:
$$\mathsf{snowy}(M) = \mathsf{snowy}(M-1) \:\bigsqcup\: \left( \bigsqcup_{i \geq 1} \mathsf{snowy}(M)_i\right).$$

Next, we define a map
\begin{align*}
    \phi: \mathsf{snowy}(M-1) &\to \mathsf{snowy}(M)_1 \\
    (\alpha_1, \alpha_2, \dots) &\mapsto (M, \alpha_1, \alpha_2, \dots)
\end{align*}
It is straightforward to see that $\phi$ is a bijection.
Furthermore, we have $\raj(\phi(\alpha)) = \raj(\alpha) + M$.
To get $\mathsf{snowy}(M)_i$ for $i>1$, notice that 
the operator $s_i$ on the set of weak compositions is a 
bijection between $\mathsf{snowy}(M)_i$
and $\mathsf{snowy}(M)_{i+1}$.
For $\alpha \in \mathsf{snowy}(M)_i$, we have
$\raj(s_i(\alpha)) = \raj(\alpha) + 1$ 
by Corollary~\ref{C: snowy changes}.
Inductively, we have
$$
\sum_{\alpha \in \mathsf{snowy}(M)_i} q^{\raj(\alpha)}
= q^{M+i - 1}\sum_{\alpha \in \mathsf{snowy}(M-1)} q^{\raj(\alpha)}.
$$
Finally, 
\begin{align*}
\sum_{\alpha \in \mathsf{snowy}(M)} q^{\raj(\alpha)} & =
\sum_{\alpha \in \mathsf{snowy}(M-1)} q^{\raj(\alpha)}
\:+\: \left(\sum_{i \geq 1}q^{M+i-1}\right) \sum_{\alpha \in \mathsf{snowy}(M-1)} q^{\raj(\alpha)}\\
& = 
\left(1 + \sum_{i \geq 1}q^{M+i-1} \right) \sum_{\alpha \in \mathsf{snowy}(M-1)} q^{\raj(\alpha)} = \left(1+\frac{q^M}{1-q}\right)\prod_{m  > 0}^{M-1}\left(1 + \frac{q^m}{1 - q}\right).
\end{align*}
\end{proof}

\section{Open Problems and Future Directions}\label{sec.open}
We conclude with several open problems for future study.
In Section~\ref{subsec.shadow.diagram}, we present the connections between the
following three constructions:
\begin{itemize}
\item[-] Positions of dark clouds in $\snow(RD(w))$;
\item[-] First step of Viennot's geometric construction;
\item[-] Bumps in the first row during Schensted insertion. 
\end{itemize}

\begin{prob}
Find further connections between Viennot's geometric  construction of Schensted insertion
and $\snow(RD(w))$.
\end{prob}

\begin{prob}
Find further connections between Schensted insertion 
and $\snow(RD(w))$.
\end{prob}

The Grothendieck to Lascoux expansion, proven in~\cite{SY}, 
involves finding certain tableaux and computing their right keys.

\begin{prob}
Find a combinatorial formula for the expansion of Castelnuovo–Mumford polynomials
into top Lascoux polynomials indexed by snowy weak compositions. 
\end{prob}

Finding a combinatorial formula for the \definition{structure constants} $c^w_{u,v}$ 
for Grothendieck polynomials, defined as
\[
\fG_u \fG_v = \sum_{w} c^w_{u,v}\fG_w\,,
\] 
has been a long-standing open problem.
These coefficients have a geometric interpretation: They are the intersection 
numbers for the Schubert classes in the connective $K$-theory. 
If we consider only the top-degree terms on both sides, we get the structure 
constants for Castelnuovo–Mumford polynomials, which we denote as $\widehat{c^w_{uv}}$,
which are still non-negative integers.

\begin{prob}
Find a combinatorial formula for $\widehat{c^w_{uv}}$.
\end{prob}

The Grothendieck polynomials $\fG_w(\bm{x})$ are a specialization of the double 
Grothendieck polynomials $\fG_w(\bm{x},\bm{y})$ by setting $y_1=y_2 = \dots = 0$. 
In~\cite{KM}, Knutson and Miller introduced pipe dream rules for both
$\fG_w(\bm{x})$ and $\fG_w(\bm{x},\bm{y})$.
For Castelnuovo–Mumford polynomials $\widehat{\fG_w}(\bm{x})$, we can think they 
correspond to a subset of pipe dreams for $\fG_w(\bm{x})$.
In~\cite{PSW}, the authors proved a factorization of $\widehat{\fG_w}(\bm{x},\bm{y})$
into a $\bm{x}$-polynomial and a $\bm{y}$-polynomial, and they showed the 
the leading term is in fact,
\[
\bm{x}^{\rajcode(w)} \bm{y}^{\rajcode(w^{-1})},
\]
with coefficient $1$ by constructing a pipe dream associated with it iteratively.

\begin{prob}
Use the snow diagrams to give an explicit construction of pipe dreams for the leading term
in $\widehat{\fG_w}(\bm{x},\bm{y})$.
\end{prob}

In general, one can define a $K$-Kohnert polynomial for any diagram $D$:
\[
\kappa_D(\bm{x};\beta) := \sum_{D'\in \KKD(D)}\bm{x}^{\wt(D')} \beta^{\ex(D')}.
\]

\begin{prob}
Find characterizations of diagram $D$ such that the leading monomial of $\widehat{\kappa}_D$ is given by $\rajcode(D)$.
\end{prob}

\section{Appendix}
\label{appendix}
\begin{table}[ht]
\begin{tabular}{|l|l|}\hline
Permutation $w$ & $\fG_w$ ($\topGro_w$ in bold blue) \\ \hline
$e^\dagger$      & \cb{1}\\[1mm]
$2134^\dagger  = s_1$ & \cb{x_1} \\[1mm]
$1324^\dagger  = s_2$ & $(x_1+x_2)+\beta$\cb{x_1x_2}\\[1mm]
$2314 = s_1s_2$ & \cb{x_1x_2} \\[1mm]
$3124^\dagger  = s_2s_1$ & \cb{x_1^2}\\[1mm]
$3214^\dagger  = s_1s_2s_1$ & \cb{x_1^2x_2} \\[2mm]
$1243^\dagger = s_3$ & $(x_1+x_2+x_3)+\beta(x_1x_2+x_1x_3+x_2x_3) + \beta^2$\cb{x_1x_2x_3}\\[1mm]
$2143^\dagger = s_1s_3$ & $(x_1x_2+x_1x_3+x_1^2) + \beta(x_1x_2x_3+x_1^2x_2+x_1^2x_3)+\beta^2$\cb{ x_1^2x_2x_3}\\[1mm]
$1342 = s_2s_3$ & $(x_1x_2+x_1x_3+x_2x_3)+\beta$\cb{2x_1x_2x_3}\\[1mm]
$1423^\dagger = s_3s_2$ & $(x_1^2+x_2^2+x_1x_2)+\beta$\cb{(x_1^2x_2+x_1x_2^2)}\\[1mm]
$2341 = s_1s_2s_3$ & \cb{x_1x_2x_3} \\[1mm]
$2413^\dagger = s_1s_3s_2$ &  $(x_1x_2^2+x_1^2x_2)+\beta$\cb{x_1^2x_2^2}\\[1mm]
$3142 = s_2s_1s_3$ & $(x_1^2x_2+x_1^2x_3)+\beta$\cb{ x_1^2 x_2 x_3} \\[1mm]
$4123^\dagger = s_3s_2s_1$ & \cb{x_1^3} \\[1mm]
$1432^\dagger = s_3s_2s_3$ & $(x_1^2x_2+x_1x_2^2 + x_1^2x_3 + x_1x_2x_3 +x_2^2x_3)+\beta(x_1^2x_2^2 + 2x_1^2x_2x_3+2x_1x_2^2x_3)+\beta^2$\cb{x_1^2x_2^2x_3} \\[1mm]
$2431 = s_1s_3s_2s_3$ & $(x_1^2x_2x_3 + x_1x_2^2x_3)+\beta$\cb{x_1^2x_2^2x_3} \\[1mm]
$3241 = s_2s_1s_2s_3$ & \cb{x_1^2x_2x_3} \\[1mm]
$3412 = s_2s_1s_3s_2$ & \cb{x_1^2x_2^2} \\[1mm]
$4132^\dagger = s_3s_2s_1s_3$ & $(x_1^3x_2+x_1^3x_3)+\beta$\cb{x_1^3x_2x_3} \\[1mm]
$4213^\dagger = s_3s_2s_1s_2$ & \cb{x_1^3x_2} \\[1mm]
$3421 = s_2s_1s_3s_2s_3$ & \cb{x_1^2x_2^2x_3}\\[1mm]
$4231 = s_3s_2s_1s_2s_3$ & \cb{x_1^3x_2x_3} \\[1mm]
$4312^\dagger = s_3s_2s_1s_3s_2$ & \cb{x_1^3x_2^2} \\[1mm]
$4321^\dagger = s_3s_2s_1s_3s_2s_3$ & \cb{x_1^3x_2^2x_3} \\ \hline
\end{tabular}
\caption{Grothendieck polynomials in $S_4$. $\dagger$ refers to inverse fireworks permutations.}
\label{table.grothendieck}
\end{table}

\begin{table}[ht]
\begin{tabular}{|l|l|}\hline
Weak compositions $\alpha$ & $\fL_\alpha$ ($\topLas_\alpha$ in bold blue)\\ \hline
$(0,0,0)^\dagger$     & \cb{1} \\[1mm]
$(1,0,0)^\dagger$     & \cb{x_1}\\[1mm]
$(0,1,0)^\dagger$     & $(x_1+x_2)+\beta$ \cb{ x_1x_2}\\[1mm]
$(1,1,0)$     & \cb{x_1x_2}\\[1mm]
$(2,0,0)^\dagger$     & \cb{x_1^2}\\[1mm]
$(2,1,0)^\dagger$     &  \cb{x_1^2x_2}\\[2mm]
$(0,0,1)^\dagger$     & $(x_1+x_2+x_3) + \beta(x_1x_2+x_1x_3+x_2x_3)+\beta^2$\cb{ x_1x_2x_3} \\[1mm]
$(0,1,1)$     & $(x_1x_2+x_1x_3+x_2x_3)+\beta$\cb{ 2x_1x_2x_3}\\[1mm]
$(0,2,0)^\dagger$     & $(x_1^2+x_1x_2+x_2^2)+\beta $ \cb{(x_1^2x_2+x_1x_2^2)}\\[1mm]
$(1,0,1)$     & $(x_1x_2+x_1x_3)+\beta$\cb{ x_1x_2x_3}\\[1mm]
$(3,0,0)^\dagger$     & \cb{x_1^3}\\[1mm]
$(2,0,1)^\dagger$     & $(x_1^2x_2+x_1^2x_3)+\beta$\cb{ x_1^2x_2x_3}\\[1mm]
$(1,2,0)^\dagger$     & $(x_1x_2^2+x_1^2x_2)+\beta$\cb{ x_1^2x_2^2}\\[1mm]
$(0,2,1)^\dagger$     & $(x_1^2x_2 + x_1^2x_3+x_1x_2^2+x_1x_2x_3+x_2^2x_3)+\beta(x_1^2x_2^2+2x_1^2x_2x_3+2x_1x_2^2x_3)+\beta^2$\cb{ x_1^2x_2^2x_3}\\[1mm]
$(1,1,1)$     & \cb{x_1x_2x_3}\\[1mm]
$(3,1,0)^\dagger$     & \cb{x_1^3x_2} \\[1mm]
$(3,0,1)^\dagger$     & $(x_1^3x_2+x_1^3x_3)+\beta$\cb{x_1^3x_2x_3}\\[1mm]
$(2,2,0)$     & \cb{x_1^2x_2^2} \\[1mm]
$(2,1,1)$     & \cb{x_1^2x_2^2} \\[1mm]
$(1,2,1)$     & $(x_1^2x_2x_3+x_1x_2^2x_3)+\beta$\cb{x_1^2x_2^2x_3}\\[1mm]
$(3,2,0)^\dagger$     & \cb{x_1^3x_2^2} \\[1mm]
$(3,1,1)$     & \cb{x_1^3x_2x_3} \\[1mm]
$(2,2,1)$     & \cb{x_1^2x_2^2x_3} \\[1mm]
$(3,2,1)^\dagger$     & \cb{x_1^3x_2^2x_3} \\ \hline
\end{tabular}
\caption{Lascoux polynomials in $C_4$. $\dagger$ refers to snowy weak compositions.}
\label{table.lascoux}
\end{table}
\bibliographystyle{alpha}
\bibliography{revision}{}
\end{document}